\documentclass[a4paper,12pt]{article}

\usepackage[utf8]{inputenc}
\usepackage[english]{babel}
\usepackage{hyperref}
\usepackage{amsmath,amsthm,enumerate}
\usepackage{enumitem} 
\usepackage{amssymb}
\usepackage{graphicx, color}
\usepackage{epsfig}
\usepackage{tikz}
\usetikzlibrary{calc,decorations.pathmorphing,decorations.text}
\usepackage{subcaption}
\usepackage{refcount}
\usepackage[hmargin=2.5cm,vmargin=3cm]{geometry}
\usepackage{mathtools, braket}
\usepackage{authblk}
\usepackage{centernot}

\theoremstyle{plain}
\newtheorem{thm}{Thm}[section]

\newtheorem{claim}[thm]{Claim}
\newtheorem{theorem}[thm]{Theorem}
\newtheorem{lemma}[thm]{Lemma}
\newtheorem{corollary}[thm]{Corollary}
\newtheorem{proposition}[thm]{Proposition}
\newtheorem{conjecture}[thm]{Conjecture}

\newtheorem{observation}[thm]{Observation}

\newenvironment{proof*}{\noindent\emph{Proof of the claim:}}{\hfill$\Diamond$}

\newcommand{\spto}{\ensuremath{\stackrel{s.p.}{\longrightarrow}}}

\newcommand{\Four}{\!\scriptscriptstyle -4}
\newcommand{\CK}{\!\scriptscriptstyle -k}
\newcommand{\CL}{\!\scriptscriptstyle -l}

\usepackage[color=green!30]{todonotes}

\usepackage{caption}
\usetikzlibrary{matrix}
\usetikzlibrary{decorations.markings}
\tikzstyle{vertex}=[circle, draw, fill=black!50,
inner sep=0pt, minimum width=4pt]
\tikzset{->-/.style={decoration={
			markings,
			mark=at position .5 with {\arrow{>}}},postaction={decorate}}}

\tikzstyle{bigblue}=[color=blue, very thick, >=stealth]
\tikzstyle{lightblue}=[color=blue, thin, >=stealth]

\tikzstyle{bigred}=[color=red, very thick, >=stealth]
\tikzstyle{lightred}=[color=red, thin, >=stealth]

\tikzstyle{biggreen}=[color=black!30!green, very thick, >=stealth]
\tikzstyle{lightgreen}=[color=black!30!green,  thin, >=stealth]

\title{Density of $C_{\Four}$-critical signed graphs}

\date{\today} 

\author{Reza Naserasr} 
\author{Lan Anh Pham}
\author{Zhouningxin Wang}

\affil{Universit\'{e} de Paris, IRIF, CNRS, F-75006, Paris, France. 
	
	Email addresses: \{reza, lananh, wangzhou4\}@irif.fr.}

\begin{document}
\baselineskip 0.65cm

\maketitle

\abstract{ A signed bipartite (simple) graph $(G, \sigma)$ is said to be $C_{\Four}$-critical if it admits no homomorphism to $C_{\Four}$ (a negative 4-cycle) but each of its proper subgraphs does. To motivate the study of $C_{\Four}$-critical signed graphs, we show that the notion of 4-coloring of graphs and signed graphs is captured, through simple graph operations, by the notion of homomorphism to $C_{\Four}$. In particular, the 4-color theorem is equivalent to: Given a planar graph $G$, the signed bipartite graph obtained from $G$ by replacing each edge with a negative path of length 2 maps to $C_{\Four}$. 

We prove that, except for one particular signed bipartite graph on 7 vertices and 9 edges, any $C_{\Four}$-critical signed graph on $n$ vertices must have at least $\lceil\frac{4n}{3}\rceil$ edges. Moreover, we show that for each value of $n\geq 9$ there exists a $C_{\Four}$-critical signed graph on $n$ vertices with either $\lceil\frac{4n}{3}\rceil$ or $\lceil\frac{4n}{3}\rceil+1$ many edges.

As an application, we conclude that all signed bipartite planar graphs of negative girth at least $8$ map to $C_{\Four}$. Furthermore, we show that there exists an example of a signed bipartite planar graph of girth $6$ which does not map to $C_{\Four}$, showing $8$ is the best possible and disproving a conjecture of Naserasr, Rollova and Sopena.}


\section{Introduction}

A \emph{homomorphism} of a graph $G$ to a graph $H$ is a mapping of the vertices of $G$ to the vertices of $H$ such that adjacencies are preserved. The theory of graph homomorphism is a natural extension of the notion of \emph{proper coloring} where a proper $k$-coloring (of a graph $G$) can be viewed as a homomorphism (of $G$) to $K_k$. One of the key concepts in the study of proper coloring is the notion of \emph{$k$-critical graphs}. A $k$-critical graph is defined as a graph of chromatic number $k$ all whose proper subgraphs are $(k-1)$-colorable. An extension of the notion to homomorphism was proposed in 1980's by Catlin~\cite{Ca88}, but the concept had not drawn much attention until recently. Given a graph $H$, a graph $G$ is said to be \emph{$H$-critical}, if $G$ does not admit a homomorphism to $H$ but each of its proper subgraphs does.

Next to the complete graphs, the most studied graphs in the theory of homomorphism are odd cycles. It is a folklore fact that the $C_{\!\scriptscriptstyle \, 2k+1}$-coloring problem captures the $(2k+1)$-coloring problem via a basic graph operation: Given a graph $G$, let $T'_{2k-1}(G)$ be the graph obtained from $G$ by subdividing each edge into a path of length $2k-1$. Then $T'_{2k-1}(G)$ admits a homomorphism to $C_{\!\scriptscriptstyle \, 2k+1}$ if and only if $G$ is properly $(2k+1)$-colorable (see \cite{HN90}). 

One of the key directions of the study of $k$-critical graphs is to bound from below the number of edges as a function of $k$ and $n$ (the number of vertices).  Kostochka and Yancey gave a nearly tight lower bound in \cite{KY14}, almost settling a conjecture of Gallai. Observing that being a $4$-critical graph is the same as being a $C_{\!\scriptscriptstyle \, 3}$-critical graph, it follows from the special case presented in \cite{KY14C} that any $C_{\!\scriptscriptstyle \, 3}$-critical graph on $n$ vertices has at least $\lceil\frac{5n-2}{3}\rceil$ edges. Their approach is extended to the study of $C_{\!\scriptscriptstyle \, 5}$-critical graphs in \cite{DP17} and to $C_{\!\scriptscriptstyle \, 7}$-critical graphs in \cite{PS19}. In \cite{DP17}, it is proved that any $C_{\!\scriptscriptstyle \, 5}$-critical graph on $n$ vertices has at least $\lceil \frac{5n-2}{4} \rceil$ edges and they conjecture that the bound can be improved to $\lceil \frac{14n-9}{11} \rceil$. Similarly, in \cite{PS19}, it is proved that any $C_{\!\scriptscriptstyle \, 7}$-critical graph on $n$ vertices has at least $\lceil \frac{17n-2}{15} \rceil$ edges and they conjecture that the bound can be improved to $\lceil \frac{27n-20}{23} \rceil$. 

In this work, based on recent development of the theory of homomorphisms of signed graphs, we show that by replacing odd cycles with negative cycles we can fill the parity gap in this study. Then focusing on $C_{\Four}$-critical signed graphs, we show that any such signed graph on $n$ vertices must have at least $\lceil\frac{4n}{3}\rceil$ edges with a sole exception of a signed bipartite graph on 7 vertices which has only 9 edges.  

In the next section, we present the necessary terminology and the relation between colorings of graphs and homomorphisms of signed graphs to negative cycles. In Section~\ref{sec:C4Critical}, we prove our main result which is on the minimum number of edges of $C_{\Four}$-critical signed graphs. In Section~\ref{sec:construction} we introduce some techniques to build $C_{\Four}$-critical signed graphs of low edge-density which we use to conclude the tightness of our bound. Finally, in Section~\ref{sec:Planar}, we consider applications to the planar case and the relation to a bipartite analogue of Jaeger-Zhang conjecture, and discuss further directions of study.

\section{Signed graphs and homomorphisms}\label{sec:Hom}
A \emph{signed graph} $(G, \sigma)$ is a graph $G$ together with an assignment $\sigma$ of signs (i.e. $+$ or $-$) to the edges of $G$. The assignment $\sigma$ is referred to as the \emph{signature}.  When the signature is not of high importance, we may write $\hat{G}$ in place of $(G, \sigma)$. When all edges are positive (resp. negative) we write $(G, +)$ (respectively $(G,-)$). When drawing a signed graph, we use solid or blue lines to represent positive edges and dashed or red lines to represent negative edges. For underlying graphs (with no signature) we use the color gray. A signed graph $(H, \sigma')$ is said to be a \emph{subgraph} (an \emph{induced subgraph}) of $(G, \sigma)$ if $H$ is a subgraph (an induced subgraph) of $G$ and $\sigma'$ is a signature on $H$ such that for every $e \in E(H)$, we have that $\sigma'(e)=\sigma(e)$. For simplicity and with a little abuse of notation, we may write $(H, \sigma)$ in place of $(H, \sigma')$.

A \emph{switching} of a signed graph $(G, \sigma)$ at a vertex $x$ is the operation of multiplying the signs of all edges incident to $x$ by a $-$.
A \emph{switching} of $(G, \sigma)$ is a collection of switchings at each of the elements of a given set $X$ of vertices. That is equivalent to switching the signs of all edges in the edge-cut $(X, V\setminus X)$. Two signatures $\sigma_1$ and $\sigma_2$ on a graph G are said to be \emph{equivalent} if one can be obtained from the other by a switching, in which case we say $(G, \sigma_1)$ is {\em switching equivalent} to $(G, \sigma_2)$.

The sign of a structure in $(G, \sigma)$ is the product of the signs of the edges in the given structure, counting multiplicity. The sign of some structures (such as a cycle or a closed walk) is invariant under a switching, while for some other structures, such as a path, the sign may change (e.g., if a switching is done in one of the two ends of a path). Thus we may relax or restrict our use accordingly. For example, when speaking of sign of a cycle, we may refer to any equivalent signature, but when speaking of sign of a path, we are restricted to the signature at hand.
In particular, any signed cycle of length $l$ with an even number of negative edges will simply be referred to as $C_{\!\scriptscriptstyle +l}$ and when there are an odd number of negative edges it will be denoted by $C_{\CL}$. As $C_{\Four}$ is the primary subject of this work, we will use the labeling of Figure~\ref{fig:C_4} when referring to this signed graph on its own, but as a subgraph of another signed graph it will have an induced labeling.

\begin{figure}[htbp]
\centering
\begin{minipage}{.35\textwidth}
		\centering
		\begin{tikzpicture}
		[scale=.25]
			\foreach \i in {4,3,2,1}
		{
			\draw[rotate=-90*(\i-2)+45] (0, 4) node[circle, draw=black!80, inner sep=0mm, minimum size=2mm] (u_\i){\scriptsize ${u_{_{\i}}}$};
		}

		\foreach \i/\j in {1/2,2/3,3/4}
		{
			\draw  [line width=0.5mm, blue] (u_\i) -- (u_\j);
		}
		\draw  [dashed, line width=0.5mm, red] (u_4) -- (u_1);
          	\end{tikzpicture}
		\caption{$C_{\Four}$}
		\label{fig:C_4}
\end{minipage}
\end{figure}

One of the preliminary facts in the study of signed graphs is that two signatures on a graph $G$ are equivalent if and only if they induce the same set of negative cycles (see \cite{Z82}). Thus, when a class of switching equivalent signed graphs on a graph $G$ is to be considered, one may refer to a partition of cycles, or, more generally, closed walks, of $G$ into two sets: positive and negative (see \cite{NSZ21} for more). Thus we have two natural definitions of homomorphisms of signed graphs.

A (switching) \emph{homomorphism} of a signed graph $(G,\sigma)$ to a signed graph $(H,\pi)$ is a mapping of $V(G)$ and $E(G)$ respectively to $V(H)$ and $E(H)$ that preserves the adjacencies, the incidences and the signs of closed walks. When there exists a homomorphism of $(G,\sigma)$ to $(H,\pi)$, we may write $(G,\sigma) \to (H,\pi)$. We may also, equivalently, say $(G, \sigma)$ is {\em $(H, \pi)$-colorable}.

An \emph{edge-sign preserving homomorphism} of a signed graph $(G,\sigma)$ to a signed graph $(H,\pi)$ is a mapping of $V(G)$ and $E(G)$ respectively to $V(H)$ and $E(H)$ that preserves the adjacencies, the incidences, and the signs of edges. When there exists an edge-sign preserving homomorphism of $(G,\sigma)$ to $(H,\pi)$, we may write $(G,\sigma) \spto (H,\pi)$.

We note that an edge-sign preserving homomorphism is equivalent to what is known as a homomorphism of 2-edge-colored graphs in the literature.
 
The two notions of homomorphisms are connected through the following observation.

\begin{observation}\label{obs:HomDef}
Given signed graphs $(G,\sigma)$ and $(H,\pi)$, we have  $(G,\sigma) \to (H,\pi)$ if and only if there exists an equivalent signature $\sigma'$ of $\sigma$ such that $(G,\sigma') \spto (H,\pi)$.
\end{observation}

While closely related, the two notions are also fundamentally different. In particular, for our main target $C_{\Four}$, while deciding if a signed graph $(G, \sigma)$ admits a homomorphism to it is an NP-complete problem \cite{DFMOP20}, the analogue edge-sign preserving problem becomes polynomial time through a duality presented in Theorem~\ref{thm:C4dual}.

In practice, we will take the condition of Observation~\ref{obs:HomDef} as the definition. Thus a homomorphism $\phi$ of $(G,\sigma)$ to $(H,\pi)$ consists of three parts: $\phi_1: V(G)\to \{+,-\}$, which decides for each vertex $v$ whether a switching is done at $v$, $\phi_2: V(G)\to V(H)$ which decides to which vertex of $H$ the vertex $v$ is mapped to, and $\phi_3: E(G)\to E(H)$ which decides the image of each edge. However, as we will consider only simple graphs in this work, $\phi_3$ is induced by $\phi_2$ and, therefore, the mapping $\phi$ is composed of $\phi_1$ and $\phi_2$, i.e. $\phi=(\phi_1, \phi_2)$. We note that since switching at $X$ is the same as switching at $V\setminus X$, the two mappings $(\phi_1, \phi_2)$ and $(-\phi_1, \phi_2)$ are identical. 

Each of these notions leads to a corresponding notion of \emph{isomorphism}, that is a homomorphism $\phi$ where $\phi_2$ and $\phi_3$ are one-to-one and onto. This, furthermore, leads to two notions of automorphism. Thus, for example, the signed graph $C_{\Four}$, as a $2$-edge-colored graph, has only one non-trivial automorphism which is $u_1 \leftrightarrow u_4$, $u_2 \leftrightarrow u_3$. However, it is both vertex-transitive and edge-transitive with respect to the notion of  (switching) homomorphism. It will be clear from the context which notion of isomorphism or automorphism we refer to. 
Following this notion of (switching) isomorphism, if $(H, \sigma')$ is a subgraph of $(G, \sigma')$,  then we may also refer to $(H, \sigma')$ as a subgraph of $(G, \sigma)$ whenever $(G, \sigma)$ is equivalent to $(G, \sigma')$.

Given a signed graph $(G, \sigma)$ and an element $ij\in \mathbb{Z}^2_2$, we define $g_{ij}(G,\sigma)$ to be the length of a shortest closed walk $W$ whose number of negative edges modulo $2$ is $i$  and whose length modulo $2$ is $j$. When there exists no such closed walk, we define $g_{ij}(G, \sigma)=\infty$. The smaller of $g_{10}(G, \sigma)$ and $g_{11}(G, \sigma)$ is the length of a shortest negative cycle of $(G, \sigma)$ and is called the \emph{negative girth} of $(G, \sigma)$. 

By the definition of homomorphisms of signed graphs, we have the following no-homomorphism lemma.
 
\begin{lemma}\label{lem:no-hom}{\rm [The no-homomorphism lemma]}
If $(G, \sigma)\to (H, \pi)$, then $$g_{ij}(G, \sigma)\geq g_{ij}(H, \pi)$$ for each $ij\in \mathbb{Z}^2_2$. 
\end{lemma}

We note that, algorithmically, it is not difficult to determine $g_{ij}(G, \sigma)$. We refer to \cite{NSZ21} and \cite{CNS20} for more on this.

We may now define the main notion of study in this work.

\subsection{$(H, \pi)$-critical signed graphs}

Given a signed graph $(H, \pi)$, a signed graph $(G, \sigma)$ is said to be \emph{$(H, \pi)$-critical} if the followings are satisfied:

\begin{itemize}
\item $g_{ij}(G, \sigma)\geq g_{ij}(H, \pi)$ for each $ij\in \mathbb{Z}^2_2$, (conditions of the no-homomorphism lemma), 
\item $(G,\sigma) \not \to (H,\pi)$,
\item $(G',\sigma) \to (H,\pi)$ for every proper subgraph $G'$ of $G$. 
\end{itemize}

This notion captures and extends the notion of $k$-critical graphs as follows: a graph $G$ is $k$-critical if the signed graph $(G, -)$ is $(K_{k-1},-)$-critical. Here the condition of the no-homomorphism lemma implies that $G$ has no loop. The notion of $H$-critical graph is also captured by viewing $H$ as the signed graph $(H,-)$ but with a minor revision.  If $G$ is an $H$-critical graph in the sense of \cite{Ca88} and it has an odd cycle $C_{\!\scriptscriptstyle \, 2k+1}$, where $\text{odd-girth}(H) > 2k+1$, then $G$ is the odd cycle $C_{\!\scriptscriptstyle \, 2k+1}$. Our first condition then eliminates these trivial cases.

For the particular case when $(H, \pi)= C_{\CL}$, we identify two cases based on the parity of $l$:

\begin{itemize}
\item $l=2k+1$. We recall that $g_{_{10}}(C_{\CL})=g_{_{01}}(C_{\CL})=\infty$. Thus if a signed graph $(G, \sigma)$ satisfies the conditions of the no-homomorphism lemma, then it must be switching equivalent to $(G, -)$. After a switching of $(G, \sigma)$ to $(G, -)$ and $C_{\CL}$ to $(C_{\!\scriptscriptstyle \, 2k+1}, -)$, the problem is then reduced to the study of $C_{\!\scriptscriptstyle \, 2k+1}$-critical graphs (of odd-girth at least $2k+1$).

\item $l=2k$. As $g_{_{01}}(C_{\CL})=g_{_{11}}(C_{\CL})=\infty$, in order for $(G, \sigma)$ to satisfy the conditions of the no-homomorphism lemma, $G$ must, in particular, be bipartite. This is the case of main interest in this work. 
\end{itemize}

We note that in the first case, to determine if $(G, \sigma)$ is switching equivalent to $(G, -)$ can be done in polynomial time and quite efficiently, but to determine if $G\to C_{\!\scriptscriptstyle \, 2k+1}$ is an NP-complete problem. In contrast, in the second case, to find an equivalent signature under which we can map $(G, \sigma)$ to $C_{\!\scriptscriptstyle \, -2k}$ is the hard part, and given a fixed signature, we can determine, in polynomial time, if there exists an edge-sign preserving homomorphism (see Theorem~\ref{thm:C4dual} and \cite{CNS20}).

\subsection{$l$-coloring and $C_{\CL}$-coloring}\label{sec:CL}

Given a signed graph $(G, \sigma)$, we define $T_{l}(G, \sigma)$ to be the signed graph $(G_l, \pi)$ where $G_l$ is obtained from $G$ by subdividing each edge so to become a path of length $l$ and $\pi$ is an assignment of signs on the edges of $G_l$ so that the sign of the $u-v$ path, corresponding to the edge $uv\in E(G)$, is the same as $-\sigma(uv)$.

The following lemma then shows the importance of the study of $C_{\CL}$-coloring.

\begin{lemma}\label{lem:T_k-2}
A graph $G$ is $l$-colorable if and only if $T_{l-2}(G, +)$ is $C_{\CL}$-colorable. Moreover, $G$ is $(l+1)$-critical if and only if $T_{l-2}(G, +)$ is $C_{\CL}$-critical.
\end{lemma}

\begin{proof}
For the first part we consider two cases based on the parity of $l$. If $l$ is an odd number, then in $T_{l-2}(G, +)$ a cycle is negative if and only if it is of odd length. Thus, this signed graph is switching equivalent to $(G_{l-2}, -)$. Then, the problem of mapping $T_{l-2}(G, +)$ to $C_{\CL}$ is reduced to a graph homomorphism problem of mapping $G_{l-2}$ to $C_{\!\scriptscriptstyle \, l}$. The equivalence then can be easily checked and we refer to \cite{HN90} for a proof.

We now assume that $l=2k$ is an even number, in which case $T_{l-2}(G, +)$ is a signed bipartite graph. 

We first show that if $T_{l-2}(G, +)\to C_{\CL}$, then $G$ is $l$-colorable. Since this can be done independently on each connected component of $G$, we may assume $G$ is connected. Observe that, as a signed graph equipped with switching, $C_{\CL}$ is both vertex-transitive and edge-transitive.
Let $x_1,x_2, \ldots x_{2k}$ be the vertices of $C_{\CL}$ in the cyclic order. Let $X_1=\{x_1, x_3, \ldots x_{2k-1}\}$ and $X_2=\{x_2, x_4, \ldots x_{2k}\}$ be the two parts of $C_{\CL}$. Let $\phi$ be a homomorphism of $T_{l-2}(G, +)$ to $C_{\CL}$. Observe that as $G$ (and therefore $T_{l-2}(G, +)$) is connected, the mapping $\phi$ preserves the bipartition of $T_{l-2}(G, +)$. Thus we may assume, without loss of generality, that the vertices of $T_{l-2}(G, +)$ which correspond to the vertices of $G$ map to the vertices in $X_1$.  Furthermore, recall that the homomorphism $\phi$ consists of two components $\phi_1:V(T_{l-2}(G, +)) \to \{+,-\}$, and $\phi_2: V(T_{l-2}(G, +)) \to X_1\cup X_2$. Thus the restriction of $\phi$ onto $V(G)$ is a mapping to the set $\{+,-\}\times X_1$ which is of order $2k$. We claim that $\phi$ is a proper coloring of $G$. That is simply because if $\phi$ maps two adjacent vertices to the same element of $\{+,-\}\times X_1$, the negative $(l-2)$-path  that is connecting them in $T_{l-2}(G, +)$ is mapped to a negative closed walk of length at most $l-2$, but that contradicts the no-homomorphism lemma.

The converse then is easier. Assume $\chi(G)\leq 2k$ and let $\psi$ be a $2k$-coloring of $G$ where $\{+,-\}\times X_1$ is the color set. Therefore the coloring $\psi$ can be viewed as $\psi=(\psi_1,\psi_2)$ where $\psi_1: V(G)\to \{+,-\}$ and $\psi_2: V(G)\to X_1$. We claim that $\psi$ can be extended as a homomorphism of $T_{l-2}(G, +)$ to $C_{\CL}$. For any edge $uv$ in $G$, noting that $\psi(u)=\psi(v)$ is not possible because $\psi$ is a proper coloring, we consider two possibilities: 
\begin{itemize}
\item  $(\psi_1(u), \psi_2(u))=(-\psi_1(v), \psi_2(v))$. The mapping $\psi$ then has applied a switching only in one end of the $u-v$ path, and thus switches it to a positive (even) path. After identifying its end points the resulting positive even cycle can be mapped to just an edge of any sign.

\item $\psi_2(u)\neq \psi_2(v)$. The two $\psi_2(u)-\psi_2(v)$ paths in $C_{\CL}$ are even, exactly one is negative, and each has length at most $l-2$. The $u-v$ path then can be mapped to the path of the same sign where the sign is taken after applying possible switching by $\psi_1$ at its end points.
\end{itemize}
For the moreover part, first we assume $G$ is $(l+1)$-critical. We need to show that $T_{l-2}(G, +)$ is $C_{\CL}$-critical. Let $e$ be an edge of  $T_{l-2}(G, +)$ and assume it is on the path corresponding to the edge $uv$ of $G$. Then since $G$ is critical, $G-uv$ admits an $l$-coloring which can be transformed into a mapping of  $T_{l-2}(G-uv, +)$ to $C_{\CL}$. This mapping could then be extended to the remaining vertices of the corresponding $uv$-path. Conversely, assuming  $T_{l-2}(G, +)$ is $C_{\CL}$-critical, we need to show that $G$ is $(l+1)$-critical. This follows from the fact that $T_{l-2}(G-uv, +)$ is a proper subgraph of $T_{l-2}(G, +)$ for any edge $uv$ and the first part of the theorem.
\end{proof}

\begin{corollary}\label{cor:T_2}
A graph $G$ is 4-colorable if and only if $T_2(G, +)$ maps to $C_{\Four}$.  
\end{corollary}

In particular, the 4-Color Theorem can be restated as:

\begin{theorem}\label{thm:4CT-restated}{\rm [The 4CT restated]} 
	For any planar graph $G$, the signed bipartite planar graph $T_2(G, +)$ maps to $C_{\Four}$. 
\end{theorem}

Observing that, for a graph $G$, the shortest (negative) cycle in $T_2(G, +)$ is of length at least $6$, (corresponding to a triangle of $G$), and introducing the bipartite analogue of Jaeger-Zhang conjecture, Naserasr, Rollova and Sopena \cite{NRS15} conjectured that any signed bipartite planar graph whose negative girth is at least 6 maps to $C_{\Four}$. In section \ref{sec:Planar}, we disprove this conjecture. However, as an application of our work we prove that if the condition on negative girth is increased to 8, then the result holds.

\section{$C_{\Four}$-critical signed graphs}\label{sec:C4Critical}

It follows from Corollary~\ref{cor:T_2} that the $C_{\Four}$-coloring problem is an NP-complete problem (see \cite{BFHN17}, and \cite{BS18}, for more on this subject). However, when edge-sign preserving homomorphisms are considered, we have a simple duality theorem (given in \cite{CNS20} and based on  Figure~\ref{fig:C4dual}) that makes it rather easy to determine the existence of an edge-sign preserving homomorphism to $C_{\Four}$. This duality notion will be used in our proofs.

\begin{figure}[htbp]
\centering

\begin{minipage}{.25\textwidth}
		\centering
		\begin{tikzpicture}
		[scale=.25]
			\foreach \i in {4,3,2,1}
		{
			\draw[rotate=-90*(\i-2)+45] (0, 4) node[circle, draw=black!80, inner sep=0mm, minimum size=2.4mm] (v_\i){};
		}

		\draw  [line width=0.5mm, blue] (v_1) -- (v_4);
	
		\draw  [dashed, line width=0.5mm, red] (v_1) -- (v_2);

		\draw  [dashed, line width=0.5mm, red] (v_3) -- (v_4);
          	\end{tikzpicture}
		
\end{minipage}
\begin{minipage}{.1\textwidth}
		\centering
		\begin{tikzpicture}
		[scale=.25]
		\draw [very thick, -> ] (0,-2) --(5,-2);
        \draw [very thick ] (3,-1.5) --(2,-2.5);
        \draw (3, -4) node {\tiny Sign Preserving};      
         \end{tikzpicture}
\end{minipage}
\begin{minipage}{.25\textwidth}
		\centering
		\begin{tikzpicture}
		[scale=.25]
			\foreach \i in {4,3,2,1}
		{
			\draw[rotate=-90*(\i-2)+45] (0, 4) node[circle, draw=black!80, inner sep=0mm, minimum size=2mm] (u_\i){\scriptsize $u_{_{\i}}$};
		}

		\foreach \i/\j in {1/2,2/3,3/4}
		{
			\draw  [line width=0.5mm, blue] (u_\i) -- (u_\j);
		}
		\draw  [dashed, line width=0.5mm, red] (u_4) -- (u_1);
          	\end{tikzpicture}
		
\end{minipage}
\caption{$C_{\Four}$ and its edge-sign preserving dual}
		\label{fig:C4dual}
\end{figure}
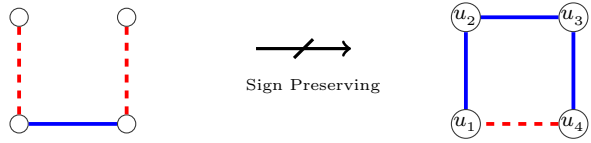

\begin{theorem}\label{thm:C4dual}\cite{CNS20}
Given a signed bipartite graph $(G, \sigma)$, we have $(G, \sigma) \spto C_{\Four}$ if and only if $(P_4, \pi) \stackrel{s.p.}{\centernot\longrightarrow} (G,\sigma)$ where $(P_4, \pi)$ is the signed path of length 3 given in Figure~\ref{fig:C4dual}.
\end{theorem}

Combined with Observation~\ref{obs:HomDef}, this theorem says that in order to map a signed bipartite graph $(G,\sigma)$ to $C_{\Four}$ it is necessary and sufficient to find a switching $(G, \sigma')$ of $(G, \sigma)$ where no positive edge is adjacent to a negative edge at each end of it.
 
It can be easily verified that any signed bipartite graph with at most two vertices on one of the two parts maps to $C_{\Four}$. Thus the first example of $C_{\Four}$-critical signed graph must have at least six vertices. Let $\Gamma$ be the signed graph obtained from $K_4$ by subdividing two nonadjacent edges, each once, with a signature assigned in such a way that each triangle of the $K_4$ become a negative 4-cycle (see Figure~\ref{fig:Gamma}). It is not hard to see that $\Gamma$ is an example of a $C_{\Four}$-critical signed graph on six vertices. In fact, up to switching, it is the unique $C_{\Four}$-critical signed graph on six vertices. We further note that $\Gamma$ has $8=\frac{4}{3}\times 6$ edges. 

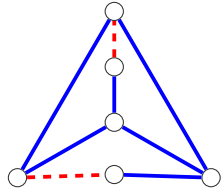
\begin{figure}[htbp]
	\centering
	\begin{minipage}[t]{.4\textwidth}
		\centering
		\begin{tikzpicture}
			[scale=.21]
					\draw (0, 0) node[circle, draw=black!80, inner sep=0mm, minimum size=2.4mm] (v_4){};
					
					\draw (0, 3.5) node[circle, draw=black!80, inner sep=0mm, minimum size=2.4mm] (v_0){};
					
					\draw (0, -3.3) node[circle, draw=black!80, inner sep=0mm, minimum size=2.4mm] (v_5){};
					
					\foreach \i in {1,2,3}
					{
						\draw[rotate=120*(\i+1)] (0,7) node[circle, draw=black!80, inner sep=0mm, minimum size=2.4mm] (v_\i){};
					}
				
					\foreach \i/\j in {1/4,1/2,1/5,0/4,3/4,2/3}
				{
					\draw[line width=0.5mm, blue] (v_\i) -- (v_\j);
				}
						\foreach \i/\j in {0/2, 3/5}
		     	{
			    	\draw[dashed, line width=0.5mm, red] (v_\i) -- (v_\j);
		     	}
		\end{tikzpicture}
	\end{minipage}
		\caption{The smallest $C_{\Four}$-critical signed graph $\Gamma$}
      \label{fig:Gamma}
\end{figure}

An example of higher interest, which is also a signed graph on a subdivision of $K_4$, is the signed graph $\hat{W}$ of Figure~\ref{fig:W} which is depicted in two different ways. This signed graph is proved in \cite{CNS20} to have smallest maximum average degree among all signed bipartite graphs that does not map to $C_{\Four}$, that is an average degree of $\frac{18}{7}$. Using the extended notion of critical signed graphs we introduced here, we will prove $\hat{W}$ to be the sole exception among the signed bipartite graphs of average degree less than $\frac{8}{3}$.

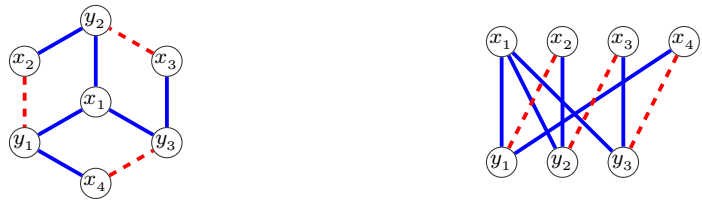
\begin{figure}[htbp]
\centering
\begin{minipage}{.4\textwidth}
		\centering
		\begin{tikzpicture}
		[scale=.24]
			\foreach \i in {2,3,4}
		{
			\draw[rotate=-120*(\i+1)+60] (0, 4.5) node[circle, draw=black!80, inner sep=0mm, minimum size=4mm] (x_\i){\scriptsize$x_{_\i}$};
		}

	\foreach \i in {1,2,3}
		{
			\draw[rotate=-120*(\i+1)] (0, 4.5) node[circle, draw=black!80, inner sep=0mm, minimum size=4mm] (y_\i){\scriptsize$y_{_\i}$};
		}

		\foreach \i in {1}
		{
			\draw[rotate=60*(\i+1)] (0, 0) node[circle, draw=black!80, inner sep=0mm, minimum size=4mm] (x_\i){\scriptsize $x_{_{\i}}$};
		}

		\foreach \i/\j in {1/1, 1/2, 1/3}
		{
			\draw  [line width=0.5mm, blue] (x_\i) -- (y_\j);
		}
		\foreach \i/\j in {2/1,3/2,4/3}
		{
			\draw  [dashed, line width=0.5mm, red] (x_\i) -- (y_\j);
		}
		\foreach \i/\j in {2/2,3/3,4/1}
		{
			\draw  [line width=0.5mm, blue] (x_\i) -- (y_\j);
		}
          	\end{tikzpicture}
		\end{minipage}
\begin{minipage}{.4\textwidth}
		\centering
		\begin{tikzpicture}
		[scale=.2]
		\foreach \i in{1,2,3,4}
		{
		\draw(4*\i, 0) node[circle, draw=black!80, inner sep=0mm, minimum size=4mm] (x_\i){\scriptsize $x_{_{\i}}$};
		}
		\foreach \i in{1,2,3}
		{
		\draw(4*\i, -8) node[circle, draw=black!80, inner sep=0mm, minimum size=4mm] (y_\i){\scriptsize $y_{_{\i}}$};
		}
		
		\foreach \i/\j in {1/1, 1/2, 1/3, 2/2, 3/3, 4/1}
		{
			\draw  [bend left=18, line width=0.5mm, blue] (x_\i) -- (y_\j);
		}
		
		\foreach \i/\j in {2/1, 3/2, 4/3}
		{
			\draw  [bend left=18, dashed, line width=0.5mm, red] (x_\i) -- (y_\j);
		}	
          	\end{tikzpicture}
	\end{minipage}
\caption{$C_{\Four}$-critical signed graph $\hat{W}$ depicted in two ways}
		\label{fig:W}
\end{figure}

We give two different proofs for the fact that $\hat{W}$ does not map to $C_{\Four}$. Each proof takes advantage of one of the presentations in Figure~\ref{fig:W}, and leads to different development of ideas. 

\begin{proposition}\label{prop:W^NoMapC4}
The signed graph $\hat{W}$ of Figure~\ref{fig:W} does not map to $C_{\Four}$. Moreover, up to a switching equivalence, this is the only signature on this graph with this property.  
\end{proposition}

\begin{proof}
 Based on the presentation on the left side, if $\hat{W}$ maps to $C_{\Four}$, then the outer 6-cycle, as it is a negative cycle, must map surjectively to $C_{\Four}$. It then follows that $x_1$ must be identified with one of $x_2, x_3, x_4$, thus creating a negative cycle of length 2 and, therefore, contradicting the no-homomorphism lemma. 

The equivalence class of signatures on this graph is determined by the signs of its three facial 4-cycles as depicted in the left side of the figure. If one of these facial 4-cycles is positive, then a degree 2 vertex on this face can be mapped to $x_1$, after a switching if needed. The resulting image then easily maps to $C_{\Four}$.  
\end{proof}

{\bf An alternative proof.} Based on the presentation on the right side, observe that each pair among $y_1, y_2, y_3$ is connected by a positive $2$-path (through $x_1$) and by a negative 2-path. Thus identifying any two of them would create $C_{\!\scriptscriptstyle -2}$. In other words, in any homomorphic image of $\hat{W}$ which is a signed simple graph, the vertices $y_1$, $y_2$ and $y_3$ must have distinct images. 

For the moreover part, given a signature on $W$ we may switch it so that $x_1y_1$, $x_1y_2$, and $x_1y_3$ are positive edges. After such a switching, if each of $x_2$, $x_3$, and $x_4$ is incident to one positive and one negative edge, then we have a (switching) isomorphic copy of $\hat{W}$. Otherwise, one of the vertices $x_2$, $x_3$, and $x_4$ can map to $x_1$. After such a mapping, we have a signed  bipartite graph on six vertices. If this signed graph does not map to $C_{\Four}$, then it must contain $\Gamma$ as a subgraph. However, $\Gamma$ has eight edges while our six-vertex graph has only seven edges. $\hfill \Box$

\medskip

Our goal here is to study the edge-density of $C_{\Four}$-critical signed graphs. We show that, with the exception of $\hat{W}$, every such signed graph has edge density at least $\frac{4}{3}$. In our proof then not only $\hat{W}$ but some constructions based on $\hat{W}$ will be of importance. 

Automorphisms of $\hat{W}$ split its vertices to three orbits: $\{x_1\}$, $\{y_1, y_2, y_3 \}$ and $\{x_2, x_3, x_4\}$ and split its edges to two orbits: those incident to $x_1$ and those on the outer 6-cycle. We will need to consider two signed graphs obtained from $\hat{W}$ by subdividing one of its edges twice and then assigning a signature on the edges of this path so that the sign of the path is the same as the sign of the edge it has replaced. Since there are two orbits of the edges on $\hat{W}$, essentially we have only two signed graphs obtained in this way. Presentations of these two signed graphs, each after a  switching, are given in Figures~\ref{fig:Omega1} and \ref{fig:Omega2}. The signed graph of Figure~\ref{fig:Omega1}, $\Omega_1$, is obtained from $\hat{W}$ by subdividing the edge $x_1y_3$ twice (where all three edges are assigned positive signs) and then switching at the vertex set $\{x_2, x_3, y_3\}$. The signed graph of Figure~\ref{fig:Omega2}, $\Omega_2$, is obtained from $\hat{W}$ by subdividing the edge $x_4y_1$ twice  (where all three edges are assigned positive signs) and then switching at the vertex set $\{x_2, x_4, y_2\}$.

\begin{figure}[htbp]
\centering
\begin{minipage}{.4\textwidth}
		\centering
		\begin{tikzpicture}
		[scale=.2]
		\foreach \i in {1,2,3,4}
		{
		\draw(4*\i, 0) node[circle, draw=black!80, inner sep=0mm, minimum size=4mm] (x_\i){\scriptsize $x_{_\i}$};
		}
		\draw(0, 0) node[circle, draw=black!80, inner sep=0mm, minimum size=4mm] (x_0){\scriptsize $x_{_0}$};
		
		\foreach \i in {1,2,3}
		{
		\draw(4*\i, -8) node[circle, draw=black!80, inner sep=0mm, minimum size=4mm] (y_\i){\scriptsize $y_{_\i}$};
		}
		
		\draw(0, -8) node[circle, draw=black!80, inner sep=0mm, minimum size=4mm] (y_0){\scriptsize$y_{_0}$};
		
		\foreach \i/\j in {0/0, 1/0, 1/1, 1/2, 2/1, 3/2, 3/3, 4/1, 4/3}
		{
			\draw  [bend left=18, line width=0.5mm, blue] (x_\i) -- (y_\j);
		}
		
		\foreach \i/\j in {2/2, 0/3}
		{
			\draw  [bend left=18, dashed, line width=0.5mm, red] (x_\i) -- (y_\j);
		}	
          	\end{tikzpicture}
		\caption{$\Omega_1$}
		\label{fig:Omega1}
\end{minipage}
\begin{minipage}{.4\textwidth}
		\centering
		\begin{tikzpicture}
		[scale=.2]
		\foreach \i in{0,1,2,3,4}
		{
		\draw(4*\i, 0) node[circle, draw=black!80, inner sep=0mm, minimum size=4mm] (x_\i){\scriptsize $x_{_\i}$};
		}
		\foreach \i in{0,1,2,3}
		{
		\draw(4*\i, -8) node[circle, draw=black!80, inner sep=0mm, minimum size=4mm] (y_\i){\scriptsize $y_{_\i}$};
		}
		
		\foreach \i/\j in {0/0,0/1, 1/1, 1/3, 2/1, 2/2, 3/2, 3/3, 4/3}
		{
			\draw  [bend left=18, line width=0.5mm, blue] (x_\i) -- (y_\j);
		}
		
		\foreach \i/\j in {4/0, 1/2}
		{
			\draw  [bend left=18, dashed, line width=0.5mm, red] (x_\i) -- (y_\j);
		}	
          	\end{tikzpicture}
		\caption{$\Omega_2$}
		\label{fig:Omega2}
\end{minipage}
\end{figure}

It is easily observed that each of the two signed graphs with the signature presented in the Figures~\ref{fig:Omega1} and \ref{fig:Omega2} satisfies the conditions of Theorem~\ref{thm:C4dual}, and, therefore, each of them maps to $C_{\Four}$. 
In the next two lemmas, we show that one cannot make either of these two signed graphs $C_{\Four}$-critical by only adding a vertex of degree $2$.

\begin{lemma}\label{lem:Omega1}
Let $\Omega_1$ be the signed graph of Figure~\ref{fig:Omega1}. If we add a vertex $v$ to one part of $\Omega_1$ and connect it with two vertices in the other part (with any signature), the resulting signed graph admits a homomorphism to $C_{\Four}$.
\end{lemma}

\begin{proof}
Let $\Omega_1$ be the signed bipartite graph of Figure~\ref{fig:Omega1} consisting of a bipartition $(X, Y)$ where $X=\{x_0, x_1, x_2, x_3, x_4\}$ and $Y=\{y_0, y_1, y_2, y_3\}$.

If the two edges incident to the new vertex $v$ are of the same sign, by switching at that new vertex, if needed, we consider them both positive. The resulting signed graph has a signature satisfying Theorem~\ref{thm:C4dual}, therefore, maps to $C_{\Four}$. Hence we assume that the two edges incident to $v$ are of different signs and consider two cases depending on to which part the vertex $v$ belongs.

\noindent
\textbf{Case 1.}  $v$ is added to the $X$ part. We consider three possibilities.

\begin{itemize}
\item $v$ is adjacent to $y_3$. By switching at $v$, if necessary, we assume that $vy_3$ is negative. The only possible problem against Theorem~\ref{thm:C4dual} is by the positive edge $vy_2$. In that case, to resolve the issue, we apply a switching at $x_2$. 

\item $v$ is not adjacent to $y_3$ but $v$ is adjacent to $y_2$. We consider $vy_2$ to be negative and we are done.

\item $v$ is adjacent to both of $y_0$ and $y_1$.  We take $vy_1$ as a negative edge and we are done after a switching at $x_2$.
\end{itemize}

\noindent
\textbf{Case 2.} $v$ is added to the $Y$ part. We consider three possibilities.
\begin{itemize}

\item $v$ is adjacent to one or both of $x_0$ and $x_2$. We switch at one of $x_0$ and $x_2$, one which is adjacent to $v$. Then by a switching at $v$ (if needed) we have both edges incident to $v$ of positive signs. The resulting signed graph satisfies the conditions of Theorem~\ref{thm:C4dual}.

\item $v$ is adjacent to $x_3$ but not adjacent to $x_0$ and $x_2$. We assume $vx_3$ is a negative edge. We switch at $x_3$ and the conditions of Theorem~\ref{thm:C4dual} are satisfied.

\item $v$ is adjacent to both of $x_1$ and $x_4$. Similarly, we assume $vx_4$ is negative. We switch at $x_0$ and we are done. \qedhere
\end{itemize}

\end{proof}

\begin{lemma}\label{lem:Omega2}
Let $\Omega_2$ be the signed graph of Figure~\ref{fig:Omega2}. If we add a vertex $v$ to one part of $\Omega_2$ and connect it with two vertices in the other part (with any signature), the resulting signed graph either contains $\hat{W}$ and maps to it or admits a homomorphism to $C_{\Four}$.
\end{lemma}

\begin{proof}
Let $\Omega_2$ be the signed bipartite graph of Figure~\ref{fig:Omega2} consisting of a bipartition $(X, Y)$ where $X=\{x_0, x_1, x_2, x_3, x_4\}$ and $Y=\{y_0, y_1, y_2, y_3\}$.

As in the previous lemma, we can assume that of the two edges incident to $v$ exactly one is negative. Again, we consider two cases depending on to which part $v$ belongs.

\noindent
\textbf{Case 1.} $v$ is added to the $X$ part. We consider three possibilities.

\begin{itemize}
\item $v$ is adjacent to $y_2$. By a switching at $v$, if needed, we assume  $vy_2$ is negative. The only reason Theorem~\ref{thm:C4dual} may not work is by the positive edge $vy_0$. This can be taken care of by switching at $y_0$.

\item $v$ is adjacent to $y_0$ but not adjacent to $y_2$.  By considering $vy_0$ as the negative edge incident to $v$, the resulting signed graph satisfies the condition of Theorem~\ref{thm:C4dual}.

\item $v$ is adjacent to both $y_1$ and $y_3$. We may assume $vy_3$ is a negative edge. The subgraph induced by $x_1,x_2,x_3, y_1, y_2, y_3$ and $v$ is (switching) isomorphic to $\hat{W}$.  To see this isomorphism, using labeling of $W$ as in the Figure~\ref{fig:W}, it is enough to switch at $x_2$ and $y_2$, and relabel $v$ as $x_4$ while keeping all other labels the same. Finally, to see that the full graph of this case maps to $\hat{W}$, it is enough to extend the previous isomorphism (which we name $\psi$) to a mapping. This is done, for example, by setting $\psi(x_0)=(+,x_1)$ and $\psi(y_0)=(+,y_2)$ and $\psi(x_4)=(+,x_3)$.
\end{itemize}

\noindent
\textbf{Case 2.} $v$ is added to the $Y$ part. We consider four possibilities.

\begin{itemize}
\item $v$ is adjacent to $x_1$.  We choose $vx_1$ to be negative.  The only obstacle against Theorem~\ref{thm:C4dual} then can come from the positive edge $vx_4$, but we can switch at $y_0$ to resolve this issue.

\item $v$ is adjacent to $x_4$ but not adjacent to $x_1$. We choose $vx_4$ to be negative. Then we already have a signature satisfying conditions of Theorem~\ref{thm:C4dual}.

\item $v$ is adjacent to $x_0$ but to neither of $x_1$ and $x_4$. Assuming that $vx_0$ is negative, we can switch at $y_0$ to apply Theorem~\ref{thm:C4dual}.

\item $v$ is adjacent to both $x_2$ and $x_3$. We assume $vx_3$ is negative. The subgraph induced by $x_1,x_2,x_3, y_1, y_2, y_3$ and $v$ is (switching) isomorphic to $\hat{W}$. One such isomorphism $\phi$ is defined as follows:
$\phi(x_1)=(-,y_2)$,
$\phi(x_2)=(+,y_1)$,
$\phi(x_3)=(+,y_3)$,
$\phi(y_1)=(-,x_2)$,
$\phi(y_2)=(+,x_1)$,
$\phi(y_3)=(+,x_3)$,
$\phi(v)=(+,x_4)$.
To complete this isomorphism  to a homomorphism of the full graph to $\hat{W}$ we map $x_0$, $x_4$ and $y_0$ as follows:
$\phi(x_0)=(-,y_2)$,
$\phi(x_4)=(+,y_3)$,
$\phi(y_0)=(-,x_1)$.
\end{itemize}

We note that for a better correspondence we have used same or similar labels for vertices of the graphs. In the mappings $\psi$ and $\phi$ thus the vertices of the domains are those of the graphs we work with but the images are those of $\hat{W}$ as labeled in Figure~\ref{fig:W}. 
\end{proof}

Some general structural properties of a $C_{\Four}$-critical signed graph are as follows.

\begin{lemma}\label{lem:2connected}
Every $C_{\Four}$-critical signed graph is $2$-connected. 
\end{lemma}

\begin{proof}
This is an easy consequence of the fact that $C_{\Four}$ is vertex transitive and we leave the details as an exercise.
\end{proof}

We say a path $P$ of length $k$ in $\hat{G}$ is a \emph{$k$-thread} if all of its $k-1$ internal vertices are of degree $2$ in $G$. It is easily observed that the maximum length of a thread in an $(H, \pi)$-critical graph is bounded by a function of $(H,\pi)$. For $C_{\Four}$-critical signed graphs, we have:

\begin{lemma}\label{lem:3thread}
A $C_{\Four}$-critical signed graph $\hat{G}$ does not contain a $3$-thread.
\end{lemma}

\begin{proof}
Assume to the contrary that $G$ has a 3-thread $P=x_0x_1x_2x_3$. Recall that a $C_{\Four}$-critical signed graph is bipartite. As $x_0$ and $x_3$ are connected by a path of length $3$, they are  in different parts of $G$. Since $\hat{G}$ is $C_{\Four}$-critical, the signed graph $\hat{G'}=\hat{G}-\{x_1,x_2\}$ maps to $C_{\Four}$. 
Let $\varphi$ be such a mapping. Observe that, by Lemma~\ref{lem:2connected}, $\hat{G'}$ is connected, thus $\varphi$ preserves the bipartition of $G'$. In particular, $\varphi(x_0)$ and $\varphi(x_3)$ are in two different parts of $C_{\Four}$ and thus adjacent. We note that $\varphi$ has possibly applied switchings on some vertices of $G'$. Working with the resulting signature obtained from the same switching on $\hat{G}$, we let $\hat{P}$ be the signed graph induced on $P$. If $\hat{P}$ has the same sign as the edge $\varphi(x_0)\varphi(x_3)$, then $\varphi$ can be extended by mapping $\hat{P}$ to this edge as well. Otherwise, $\varphi$ can be extended by mapping $\hat{P}$ to the rest of the $C_{\Four}$ (that is $C_{\Four}-\varphi(x_0)\varphi(x_3)$).
\end{proof}

In this lemma, length 3 for a forbidden thread is the best one can do. We have already seen examples of $C_{\Four}$-critical signed graphs with vertices of degree $2$, that correspond to 2-threads. However, we may still apply some restriction on such threads:

\begin{observation}
\label{obs:signofpositive4cycle}
Given a $C_{\Four}$-critical signed graph, a vertex of degree 2 cannot be on a $C_{\!\scriptscriptstyle +4}$.
\end{observation}

In fact, this is generally true for every $(H,\pi)$-critical signed graph and also for any signed graph which admits no homomorphism to a subgraph of itself.

We are now ready to state and prove our main result on the structure of $C_{\Four}$-critical signed graphs.

\subsection{Edge-density of $C_{\Four}$-critical signed graphs}\label{sec:densityC4Critical}

We will use the notion of potential developed in \cite{KY14C} (and then further used in \cite{DP17} and \cite{PS19}) to prove the following.

\begin{theorem}\label{thm:densitystronger}
	If $\hat{G}$ is a $C_{\Four}$-critical signed graph that is not isomorphic to $\hat{W}$, then $$|E(G)|\geq \dfrac{4|V(G)|}{3}.$$
\end{theorem}

Thus the natural \emph{potential} function of graphs we may work with is: $$p(G)=4|V(G)|-3|E(G)|.$$ 

We note that the potential of a signed graph is the potential of  its underlying graph.

\begin{observation}\label{obs:potential4vertices}
We have $p(K_1)=4$, $p(K_2)=5$, $p(P_3)=6$ and $p(C_{\!\scriptscriptstyle \, 4})=4$. Thus any signed bipartite graph on at most $4$ vertices has potential at least $4$.
\end{observation}

In the rest of this section,  we let $\hat{G}=(G, \sigma)$ be a minimum counterexample to Theorem~\ref{thm:densitystronger}. That is to say, $\hat{G}$ is a $C_{\Four}$-critical signed graph which is not isomorphic to $\hat{W}$, it satisfies $p(\hat{G})\geq 1$, and for any signed graph $\hat{H}$, $\hat{H}\neq \hat{W}$, with $|V(\hat{H})|< |V(\hat{G})|$ satisfying $p(\hat{H})\geq 1$, $\hat{H}$ admits a homomorphism to $C_{\Four}$.

Given a signed graph $\hat{H}$, we denote a signed graph obtained from $\hat{H}$ by adding a new vertex and joining it to two vertices of $\hat{H}$ (where the signs of the two new edges are arbitrary) by $P_2(\hat{H})$. The notation $P_2(\hat{H})$ here follows previous works (\cite{DP17}, \cite{PS19})) where $2$ denotes the length of the path. To denote a path as a graph we use $P_n$ where $n$ is the number of the vertices. 

In the following lemma, we list the plausible potential of the subgraphs of the minimum counterexample $\hat{G}$.

\begin{lemma}\label{lem:SubgraphDensity}
Let $\hat{G}=(G, \sigma)$ be a minimum counterexample to Theorem~\ref{thm:densitystronger} and let $\hat{H}$ be a subgraph of $\hat{G}$. Then 
\begin{enumerate}
\item $p(\hat{H})\geq 1$ if $\hat{G}=\hat{H}$;
\item $p(\hat{H})\geq 3$ if $\hat{G}=P_2(\hat{H})$;
\item $p(\hat{H})\geq 4$ otherwise.
\end{enumerate}
\end{lemma}

\begin{proof}
The first claim is our assumption on $\hat{G}$. If $\hat{G}=P_2(\hat{H})$, then $p(\hat{G})=p(\hat{H})+4\times 1-3\times 2$, and then, since $p(\hat{G})\geq 1$, we have $p(\hat{H})\geq 3$. We now prove that for any other subgraph of $\hat{G}$, $p(\hat{H})\geq 4$.

Suppose to the contrary that $\hat{G}$ contains a proper subgraph $\hat{H}$ which does not satisfy $\hat{G}=P_2(\hat{H})$, and satisfies $p(\hat{H})\leq 3$. Among all such subgraphs, let $\hat{H}$ be chosen so that $|V(\hat{H})|+|E(\hat{H})|$ is maximized. As adding an edge to a graph only decreases the potential, the assumption of the maximality implies that $\hat{H}$ is an induced subgraph of $\hat{G}$. 

By Observation~\ref{obs:potential4vertices}, $|V(\hat{H})|\geq 5$. As $\hat{G}$ is $C_{\Four}$-critical and $\hat{H}$ is a proper  subgraph, there is a homomorphism $\varphi$ of $\hat{H}$ to $C_{\Four}$. Since $C_{\Four}$ is vertex transitive, we may assume that $\varphi$ preserves the bipartition of $\hat{H}$ induced by the bipartition of $\hat{G}$. This is automatic if $H$ is connected, but important if $H$ is not connected.
 
Observe that the mapping $\varphi$ may have applied switching on some vertices of $\hat{H}$. Applying switching on the same set of vertices of $\hat{G}$, we get a switching equivalent signed graph. For simplicity, and without loss of generality, we may assume that $\hat{G}$ was given with this signature already. In other words, we may assume, without loss of generality, that $\varphi_1(x)=+$ for every vertex $x$ of $\hat{H}$ (recall that $\varphi=(\varphi_1, \varphi_2)$).

Define $\hat{G}_1$ to be a signed (multi)graph obtained from $\hat{G}$ by first identifying vertices of $\hat{H}$ which are mapped to the same vertex of $C_{\Four}$ under $\varphi$, and then identifying all parallel edges of the same sign. Observe that $\hat{G}_1$ is a homomorphic image of $\hat{G}$ and that $\varphi(\hat{H})$ is (isomorphic to) the image of $\hat{H}$ in this mapping. Recall that in the mapping of $\hat{G}$ to $\hat{G}_1$ the bipartition is preserved. Therefore, $\hat{G}_1$ is bipartite. Since homomorphism is an associative relation, and since $\hat{G}\not\to C_{\Four}$, we have $\hat{G}_1\not\to C_{\Four}$.  This can only be for one of two reasons: Either $\hat{G}_1$ contains a $C_{\!\scriptscriptstyle -2}$, or $\hat{G}_1$ contains a $C_{\Four}$-critical subgraph. We consider the two cases separately:

\textbf{Case 1.} $\hat{G}_1$ contains a $C_{\!\scriptscriptstyle -2}$.

This implies that $\hat{G}$ contains a negative path $\hat{P}$ of length $2$ with both endpoints in $\hat{H}$ and with its internal vertex in $V(\hat{G})\setminus V(\hat{H})$. We have that 
\begin{equation}\label{equ:1}
p(\hat{H}+\hat{P})=p(\hat{H})+4\times 1-3\times 2=p(\hat{H})-2< p(\hat{H}).
\end{equation}
Recall that $\hat{H}$ is a maximum proper subgraph satisfying that $\hat{G}\neq P_2(\hat{H})$ and $p(\hat{H})\leq 3$. Noting that $\hat{H}\subsetneq \hat{H}+\hat{P}$ and $\hat{H}+\hat{P}$ is a subgraph of $\hat{G}$, and by the maximality of $\hat{H}$, there are two possibilities: either $\hat{H}+\hat{P}$ is not a proper subgraph of $\hat{G}$, i.e., $\hat{G}=\hat{H}+\hat{P}$, or $\hat{G}=P_2(\hat{H}+\hat{P})$.
The former case is impossible as $\hat{G}\neq P_2(\hat{H})$. So $\hat{G}=P_2(\hat{H}+\hat{P})$ and then 
\begin{equation}\label{equ:2}
p(\hat{H}+\hat{P})=p(\hat{G})-4\times 1+3\times 2\geq 1-4+6=3\geq p(\hat{H}),
\end{equation} 
which is in contradiction with (\ref{equ:1}).

\textbf{Case 2.} $\hat{G_1}$ contains a $C_{\Four}$-critical subgraph $\hat{G}_2$. 

We classify the vertices of $\hat{G}_2$ into two parts: Those of the images of $V(\hat{H})$, and the remaining vertices.  We denote the former set by $X_1$, more precisely $X_1=\varphi(V(\hat{H}))\cap V(\hat{G}_2)$, and the latter set by $A$, more precisely $A=V(\hat{G}_2)\setminus X_1$.
The subgraph induced by $X_1$ is denoted by $\hat{X}$, in other words $\hat{X}=\varphi(\hat{H})\cap \hat{G}_2$. Observe that since $\hat{G}_2 \not \to C_{\Four}$ and $\varphi(\hat{H})\subset C_{\Four}$, $A\neq \emptyset$. 

Since $|V(\hat{H})|\geq 5$ and $\varphi$ is a mapping of $\hat{H}$ to $C_{\Four}$, at least two vertices are identified and, therefore, $|V(\hat{G}_2)|\leq |V(\hat{G}_1)|<|V(\hat{G})|$. As $\hat{G}$ is a minimum counterexample to Theorem~\ref{thm:densitystronger}, we have either $p(\hat{G}_2)\leq 0$ or $\hat{G}_2=\hat{W}$. Since $p(\hat{W})=1$, in all cases we have $p(\hat{G}_2)\leq 1$.

We now define a subgraph $\hat{G}_3$ of $\hat{G}$ as follows: Vertices of $\hat{G}_3$ are those vertices of $\hat{G}$ each of which is either a vertex of $\hat{G}_2$ or a vertex of $\hat{H}$. That is to say $V(\hat{G}_3) =A\cup V(\hat{H})=\{V(\hat{G}_2) \cup V(\hat{H})\} \setminus X_1$. To give the edge set of $\hat{G}_3$, we first choose a set $E'$ of the edges of $\hat{G}$ as follows: If a vertex $u\in A$ is adjacent to a vertex $v\in X_1$, then we choose a vertex $v'\in V(\hat{H})$ such that first of all $\varphi(v')=v$, second of all $uv'\in E(G)$. From the construction of $\hat{G}_1$, it is clear that there is such a vertex $v'$. We note that, there might be more than one choice for $v'$, in which case we select exactly one at random, and then let $uv'$ be an edge in $E'$. The edge set of $\hat{G}_3$ is then defined to be the set of edges of $\hat{G}$ that are either induced by $A$, or by $V(H)$ or edges in $E'$, with signature induced from the fixed signature of $\hat{G}$. In other words, $E(\hat{G}_3)=E(\hat{G}_2- \hat{X}) + E(\hat{H}) + E'$. Since each connection between the vertices of $\hat{X}$ and $\hat{G}_2-\hat{X}$ has a unique corresponding edge in $E'$, it follows that $|E(\hat{G}_3)|=|E(\hat{G}_2)|-|E(\hat{X})|+|E(\hat{H})|$ and, therefore, $p(\hat{G}_3)=p(\hat{G}_2)-p(\hat{X})+p(\hat{H})$. 

Since $\hat{G}$ and $\hat{G}_2$ are both $C_{\Four}$-critical signed graphs, $\hat{G}_2$ is not a subgraph of $\hat{G}$, that is to say $\hat{X}\neq \emptyset$. As $\hat{X}$ is a subgraph of $C_{\Four}$, by Observation~\ref{obs:potential4vertices}, $p(\hat{X})\geq 4$. Then we obtain that 
\begin{equation}\label{equ:3}
p(\hat{G}_3)=p(\hat{G}_2)-p(\hat{X})+p(\hat{H})\leq1-4+p(\hat{H})=p(\hat{H})-3\leq 0.
\end{equation}
Since $\hat{G}_3$ is a subgraph of $\hat{G}$ and $\hat{H}\subsetneq \hat{G}_3$ (because $A\neq \emptyset$), by the maximality of $\hat{H}$ and noting that $p(\hat{G}_3)< p(\hat{H})$,  either $\hat{G}=\hat{G}_3$ or $\hat{G}=P_2(\hat{G}_3)$.
If $\hat{G}=\hat{G}_3$, then $p(\hat{G}_3)\geq 1$; if $\hat{G}=P_2(\hat{G}_3)$, then $p(\hat{G}_3)\geq 3$, each of which is contradicting (\ref{equ:3}).
\end{proof}

Towards proving Theorem~\ref{thm:densitystronger}, next we show that the underlying graph $G$ of the minimum counterexample $\hat{G}$ does not contain two $4$-cycles sharing edges. 

\begin{figure}[htbp]
		\centering
		\begin{minipage}[t]{.3\textwidth}
			\centering
		\begin{tikzpicture}
		[scale=.26]
		\foreach \i in {1,2,3,4}
		{
			\draw[rotate=90*(\i-1)] (0, 3.5) node[circle, draw=black!80, inner sep=0mm, minimum size=2mm] (x_\i){\scriptsize $x_{_\i}$};
		}
	\foreach \i in {0}
	{
		\draw[rotate=90*(\i)-30] (0, 0) node[circle, draw=black!80, inner sep=0mm, minimum size=2mm] (x_\i){\scriptsize $x_{_\i}$};
	}
	
		\foreach \i/\j in {2/1, 1/4, 2/3, 3/4, 2/0, 0/4}
		{
			\draw[line width=0.5mm, gray] (x_\i) -- (x_\j);
			}
		\end{tikzpicture}
		\caption{ $\Theta_1$}
		\label{fig:2C4s}
		\end{minipage}
		\begin{minipage}[t]{.3\textwidth}
		\centering
		\begin{tikzpicture}
		[scale=.3]
		\foreach \i in {1,2}
		{
			\draw[rotate=60*(\i-1)-30] (0, 3.5) node[circle, draw=black!80, inner sep=0mm, minimum size=2mm] (x_\i){\scriptsize $x_{_\i}$};
		}
		\foreach \i in {4,5}
		{
			\draw[rotate=60*(\i-1)-30] (0, 3.5) node[circle, draw=black!80, inner sep=0mm, minimum size=2mm] (x_\i){\scriptsize $x_{_\i}$};
		}
		\foreach \i in {3}
		{
			\draw[rotate=60*(\i-1)-30] (0, 3.5) node[circle, draw=black!80, inner sep=0mm, minimum size=2mm] (x_\i){\scriptsize $x_{_\i}$};
		}
		\foreach \i in {6}
		{
			\draw[rotate=60*(\i-1)-30] (0, 3.5) node[circle, draw=black!80, inner sep=0mm, minimum size=2mm] (x_\i){\scriptsize $x_{_\i}$};
		}
		\foreach \i/\j in {3/6, 1/2, 4/5, 2/3, 1/6, 3/4, 5/6}
		{
			\draw[line width=0.5mm, gray] (x_\i) -- (x_\j);
		}
		\end{tikzpicture}
		\caption{ $\Theta_2$}
		\label{fig:2C4}
	\end{minipage}
\end{figure}

\begin{claim}\label{clm:s}	
Given a minimum counterexample $\hat{G}$ to Theorem~\ref{thm:densitystronger}, the underlying graph $G$ does not contain the graph $\Theta_1$ of Figure~\ref{fig:2C4s} as a subgraph. 
\end{claim}
	
\begin{proof}
By contradiction, assume $\Theta_1$ is a subgraph of $G$ and let $x_0, x_1, \ldots, x_4$ be the labeling of its vertices in $G$ as well. Observe that $p(\Theta_1)=2$. Thus, by Lemma~\ref{lem:SubgraphDensity}, $G=(\Theta_1, \sigma)$ for some signature $\sigma$. We note that there are three $4$-cycles in $\Theta_1$, of which at least one must be a positive $4$-cycle. By Observation~\ref{obs:signofpositive4cycle} and as $d(x_0)=d(x_1)=d(x_3)=2$, no signature on $\Theta_1$ would result in a $C_{\Four}$-critical signed graph. 
\end{proof}

\begin{claim}\label{clm:theta}
Given a minimum counterexample $\hat{G}$ to Theorem~\ref{thm:densitystronger}, the underlying graph $G$ does not contain the graph $\Theta_2$ of Figure~\ref{fig:2C4} as a subgraph. 
\end{claim}

\begin{proof}

By contradiction, assume $\Theta_2$ is a subgraph of $G$ and let $x_1, x_2, \ldots, x_6$ be its vertices in $G$ as well. Observe that $p(\Theta_2)=3$, thus, by Lemma~\ref{lem:SubgraphDensity}, either $G$ has only six vertices, or it has seven vertices and $G=P_2(\Theta_2)$. By Lemma~\ref{lem:3thread}, no signature on $\Theta_2$ would result in a $C_{\Four}$-critical signed graph. If we add one or more edges to $\Theta_2$, then we will have a graph on 6 vertices and at least $\frac{4}{3}\times 6=8$ edges. This cannot form a counterexample. We note that after adding an edge one may assign a signature to get the only $C_{\Four}$-critical signed graph on six vertices $\Gamma$.

The remaining possibility is that $G=P_2(\Theta_2)$. Let $w$ be the added vertex.
If $w$ is not adjacent to one of $x_1$ or $x_2$, then we have a contradiction by Lemma~\ref{lem:3thread}. Similarly, $w$ must also be adjacent to one of $x_4$ and $x_5$. As $G$ is bipartite and by the symmetries of $\Theta_2$, we may assume that $w$ is adjacent to $x_1$ and $x_5$. Thus the underlying graph of $G$ is the same as that of $\hat{W}$, and by Proposition~\ref{prop:W^NoMapC4} it must be (switching) isomorphic to $\hat{W}$.
\end{proof}

In the next lemma, we imply further structure on the neighborhood of a 2-thread.

\begin{lemma}\label{lem:2thread-2}
Let $vv_1u$ be a $2$-thread in $\hat{G}$. Suppose that $v$ is a vertex of degree 3 and let $v_2, v_3$ be the other two neighbors of $v$. Then the path $v_2vv_3$ must be contained in a negative $4$-cycle in $\hat{G}$.
\end{lemma}

\begin{proof}
Suppose to the contrary that the path $v_2vv_3$ is not contained in a negative $4$-cycle. If needed, by switching at $v_2$ or $v_3$, we may assume that both $vv_2$ and $vv_3$ are of positive sign. Then by identifying $v_2$ and $v_3$ to a new vertex $v_0$, we get a homomorphic image $\hat{G}_1$ of $\hat{G}$. Observe that, since $v_2$ and $v_3$ are in the same part of $G$, $G_1$ is also bipartite. Furthermore, by our assumption, $\hat{G}_1$ does not contain a $C_{\!\scriptscriptstyle -2}$ and, therefore, $g_{ij}(\hat{G}_1)\geq g_{ij}(C_{\Four})$ for every $ij \in \mathbb{Z}_2^2$.

As $\hat{G}$ does not map to $C_{\Four}$, its homomorphic image, $\hat{G}_1$, does not map to it either. Thus there must be a $C_{\Four}$-critical subgraph $\hat{G}_2$ of $\hat{G}_1$. By Lemmas \ref{lem:3thread} and \ref{lem:2connected}, neither of the vertices $v$ and $v_1$ is a vertex of $\hat{G}_2$. On the other hand, $v_0 \in V(\hat{G}_2)$, as otherwise $\hat{G}_2$ is a proper subgraph of $\hat{G}$ which does not map to $C_{\Four}$, contradicting the fact that $\hat{G}$ is $C_{\Four}$-critical.
Since $\hat{G}$ is a minimum counterexample to the Theorem and $|V(\hat{G}_2)|<|V(\hat{G})|$, we have either $p(\hat{G}_2)\leq 0$ or $\hat{G}_2=\hat{W}$ in which case $p(\hat{G}_2)=1$.

Let $\hat{G}_3$ be the signed graph obtained from $\hat{G}_2$ by splitting $v_0$ back to $v_2$ and $v_3$, adding the vertex $v$ and adding the positive edges $vv_2$ and $vv_3$ back. Note that $\hat{G}_3$ is a subgraph of $\hat{G}$. We observe that
\begin{equation}\label{equ:G_2G_3}
p(\hat{G}_3)=p(\hat{G}_2)+4\times 2-3\times 2=p(\hat{G}_2)+2\leq 3.
\end{equation}
Furthermore, the equality is only possible if $\hat{G}_2=\hat{W}$. As $v_1\not\in V(\hat{G}_3)$, we know that $\hat{G}_3\neq \hat{G}$. By Lemma~\ref{lem:SubgraphDensity}, we must have $p(\hat{G}_3)=3$ and $\hat{G}=P_2(\hat{G}_3)$. And since equality in (\ref{equ:G_2G_3}) must hold, we also have $\hat{G}_2=\hat{W}$. 

As $\hat{G}_2=\hat{W}$, vertices of $\hat{G}_2$ are of degree 2 or 3, and, thus, the splitting operation on $v_0$ (that we considered in order to build $\hat{G}_3$) is the same as subdividing one of its edges twice. Recall that there are two types of edges in $\hat{W}$ up to (switching) isomorphism. Thus the subdivided signed graph $\hat{G}_3$ is one of the two signed graphs: either $\Omega_1$ of Figure~\ref{fig:Omega1} or $\Omega_2$ of Figure~\ref{fig:Omega2}.
Thus either $\hat{G}=P_2(\Omega_1)$ or $\hat{G}=P_2(\Omega_2)$. In the former case, by Lemma~\ref{lem:Omega1}, $\hat{G}$ maps to $C_{\Four}$. In the latter case, by Lemma~\ref{lem:Omega2}, either $\hat{G}$ maps to $C_{\Four}$ or it contains $\hat{W}$ as a proper subgraph but this contradicts the fact that $\hat{G}$ is $C_{\Four}$-critical.
\end{proof}

By combining Lemma~\ref{lem:2thread-2} with Claims~\ref{clm:s} and \ref{clm:theta}, we have our main forbidden configuration as follows:

\begin{corollary}\label{lem:3_2vertex}
A vertex of degree $3$ in the minimum counterexample $\hat{G}$ does not have two neighbors of degree $2$.
\end{corollary}

We are now ready to prove Theorem~\ref{thm:densitystronger}. 

\begin{proof}
(Of Theorem~\ref{thm:densitystronger})
We will employ the discharging technique. We assign an initial charge of $c(v)=d(v)$ to each vertex of $G$. Observe that the total charge is $2|E(G)|$. We apply the following discharging rule:\\

{\em ``Every vertex of degree 2 receives a charge of $\frac{1}{3}$ from each of its neighbors.'' \\}

For each vertex $v$, let $c'(v)$ be the charge of $v$ after the discharging procedure. Since there is no 3-thread in $G$, each degree 2 vertex $v$ receives a total of $\frac{2}{3}$ from its two neighbors and thus $c'(v)=2+\frac{2}{3}=\frac{8}{3}$. Each degree 3 vertex $u$ has at most one neighbor of degree $2$, so $c'(u)\geq 3-\frac{1}{3}=\frac{8}{3}$. Each vertex $w$ of degree at least $4$ has charge $c'(w)\geq d(w)-\frac{d(w)}{3}=\frac{2d(w)}{3}\geq\frac{8}{3}$. Thus the total charge is at least $\frac{8|V(G)|}{3}$. 
That contradicts the assumption that $p(\hat{G})=4|V(G)|-3|E(G)|\geq 1$. 
\end{proof}

Applying this result in terms of maximum average degree of the (underlying) graph, denoted $mad(G)$, we have the following. 

\begin{corollary}\label{cor:mad}
	Given a signed bipartite (simple) graph $\hat{G}$, if $mad(G)<\frac{8}{3}$ and $\hat{G}$ does not contain $\hat{W}$ as a subgraph, then $\hat{G}\to C_{\Four}$.
\end{corollary}

\section{Constructions of (sparse) $C_{\Four}$-critical signed graphs}\label{sec:construction}

We have already seen that $\chi(G) \leq k$ if and only if $T_{k-2}(G,+)\to C_{\CK}$. Next we extend this connection based on the notion of $0$-free coloring of signed graphs introduced by Zaslavsky in \cite{Za82}.

The notion of $0$-free coloring of signed graphs is one of the most natural extensions of the notion of proper coloring of graphs. Given the set $X_{2k}=\{\pm1, \pm2, \ldots, \pm k\}$, a signed multigraph $(G,\sigma)$ is said to be \emph{$X_{2k}$-colorable} if there exists an assignment $c:V(G)\to X_{2k}$ such that for each edge $e$ with endpoints $x$ and $y$ (not necessarily distinct), we have $$c(x)\neq \sigma(e)c(y).$$ Furthermore, $(G,\sigma)$ is said to be \emph{$X_{2k}$-critical} if it does not admit an $X_{2k}$-coloring but each of its proper subgraphs does.

Using this notion, Lemma~\ref{lem:T_k-2} can be extended to the following theorem. A proof of this theorem is also obtained by revising the proof of Lemma~\ref{lem:T_k-2} given in Section~\ref{sec:CL} and we leave the details to the reader.

\begin{theorem}\label{thm:T_k-2}
A signed multigraph $\hat{G}$ admits an $X_{2k}$-coloring if and only if $T_{2k-2}(\hat{G}) \to C_{\!\scriptscriptstyle -2k}$. Moreover, $\hat{G}$ is $X_{2k}$-critical if and only if $T_{2k-2}(\hat{G})$ is $C_{\!\scriptscriptstyle -2k}$-critical.
\end{theorem}

Given a graph $G$, let $\tilde{G}$ be the signed multigraph obtained from $G$ by replacing each edge of $G$ with a pair of edges: one of positive sign, another of negative sign. It is easily observed that: 

\begin{observation}\label{obs:k-X_2k-coloring}
A graph $G$ is $k$-colorable if and only if the signed multigraph $\tilde{G}$ is $X_{2k}$-colorable.
\end{observation}

Next we develop another technique to build $C_{\!\scriptscriptstyle -2k}$-critical signed graphs.

\begin{theorem}\label{thm:T-of-Tilde}
Given a graph $G$, we have $\chi(G)\leq k$ if and only if $T_{2k-2}(\tilde{G})\to C_{\!\scriptscriptstyle -2k}$. Moreover, $G$ is $(k+1)$-critical if and only if $T_{2k-2}(\tilde{G})$ is $C_{\!\scriptscriptstyle -2k}$-critical.
\end{theorem}

\begin{proof}
The first part of the theorem follows from Theorem~\ref{thm:T_k-2} and Observation~\ref{obs:k-X_2k-coloring}. 

For the moreover part, we first assume that
$T_{2k-2}(\tilde{G})$ is $C_{\!\scriptscriptstyle -2k}$-critical and need to show that $G$ is $(k+1)$-critical. For this it is enough to show that every proper subgraph $H$ of $G$ is $k$-colorable. Since $T_{2k-2}(\tilde{H})$ is a proper subgraph of $T_{2k-2}(\tilde{G})$ and any proper subgraph of $T_{2k-2}(\tilde{G})$ maps to $C_{\!\scriptscriptstyle -2k}$, by the first part of the theorem, $H$ is $k$-colorable.

Next we assume $G$ is $(k+1)$-critical. Let $e$ be an edge of $T_{2k-2}(\tilde{G})$. Our goal is to show that $T_{2k-2}(\tilde{G})-e$ maps $C_{\!\scriptscriptstyle -2k}$. Let $uv$ be the edge in $G$ which corresponds to the thread (of $T_{2k-2}(\tilde{G})$) to which $e$ belongs. 

Let $G'=G/uv$ be the graph obtained from $G$ by contracting the edge $uv$ and let $H=G-uv$. We observe that $G'$ is $k$-colorable because first of all $H$ is $k$-colorable, and, secondly, in any such coloring $u$ and $v$ must receive the same color.

Next we consider two signed graphs obtained from $T_{2k-2}(\tilde{H})$: (1) The signed graph $T_{uv}^+$ is obtained by identifying $u$ and $v$. (2) The signed graph $T_{uv}^-$ is obtained by switching at $u$ and then identifying $u$ and $v$.   
One may easily observe that these two signed graphs are (switching) isomorphic, and in fact each of them can be regarded as $T_{2k-2}(\tilde{G'})$. Since $G'$ is $k$-colorable, and by the first part of the theorem, $T_{2k-2}(\tilde{G'})$ maps to $C_{\!\scriptscriptstyle -2k}$.

If $e$ is deleted from the negative $uv$-thread, then we consider a mapping of $T_{uv}^+$ to  $C_{\!\scriptscriptstyle -2k}$. This mapping can also be viewed as a mapping of $T_{2k-2}(\tilde{H})$ to $C_{\!\scriptscriptstyle -2k}$ where $u$ and $v$ are identified without any switching on them. As the positive $uv$-thread is of even length $2k-2$, this mapping can easily be extended to a mapping of $T_{2k-2}(\tilde{G})-e$ to $C_{\!\scriptscriptstyle -2k}$. When $e$ is on the positive $uv$-thread, we will consider a mapping of $T_{uv}^-$ to  $C_{\!\scriptscriptstyle -2k}$. We note that after switching at $u$, the negative path becomes positive. The mapping can then be extended as in the previous case.
\end{proof}

This theorem further emphasizes on the importance of the study of $C_{\!\scriptscriptstyle -2k}$-critical graphs, and, more generally, of homomorphisms of signed bipartite graphs. The operation $T'_{2k-1}$ applied on graphs (as defined in the introduction) connects $(2k+1)$-coloring problem of graphs to $C_{\!\scriptscriptstyle \, 2k+1}$-coloring problem of graphs. Thus only odd values of the chromatic number are captured by  $C_{\!\scriptscriptstyle \, 2k+1}$-coloring problem. The operation $T_{2k-2}$, when applied on signed multigraphs $\tilde{G}$,  connects the $k$-coloring problem of $G$ to $C_{\!\scriptscriptstyle -2k}$-coloring problem of signed graphs. Thus $C_{\!\scriptscriptstyle -2k}$-coloring problem captures $k$-coloring problem for all the values of $k$. We note that $T_{2}(\tilde{G})$ is the same as $S(G)$ defined in \cite{NRS15} and refer to this reference for more on the importance of the homomorphisms of signed bipartite graphs.

By Theorem~\ref{thm:T-of-Tilde} and noting that odd cycles are the only 3-critical graphs, we have $T_{2}(\tilde{C}_{\!\scriptscriptstyle \, 2k+1})$ as an example of $C_{\Four}$-critical signed graph for each value of $k$. See Figures~\ref{fig:T2C_3} and \ref{fig:T2C_5} for $T_{2}(\tilde{C}_{\!\scriptscriptstyle \, 3})$ and $T_{2}(\tilde{C}_{\!\scriptscriptstyle \, 5})$. The signed bipartite graph $\hat{G}_{2k+1}=T_{2}(\tilde{C}_{\!\scriptscriptstyle \, 2k+1})$ has $6k+3$ vertices and $8k+4$ edges. Thus $T_{2}(\tilde{C}_{\!\scriptscriptstyle \, 2k+1})$ is an example of a $C_{\Four}$-critical signed graph for which the bound of Theorem~\ref{thm:densitystronger} is tight. 

Let $\hat{G}'_{2k+1}$ be the signed (bipartite) graph obtained from $\hat{G}_{2k+1}$ by identifying two vertices of degree $2$ which are at distance 2 and their common neighbor is adjacent to both with positive edges. See Figure~\ref{fig:T2C'_5} for an illustration of $\hat{G}'_5$. Observe that $\hat{G}'_{3}$ contains $\hat{W}$ as a proper subgraph and, thus, is not $C_{\Four}$-critical. For $k\geq2$, $\hat{G}'_{2k+1}$ does not map to $C_{\Four}$ because it is a homomorphic image of $\hat{G}_{2k+1}$. Moreover, it has average degree of $\frac{8|V(G'_{2k+1})|+2}{3|V(G'_{2k+1})|}$, it does not contain $\hat{W}$ as a subgraph and any proper subgraph of it has average degree strictly less than $\frac{8}{3}$. Thus, by Corollary~\ref{cor:mad}, it is a $C_{\Four}$-critical signed graph for which the bound of Theorem~\ref{thm:densitystronger} is tight.
Further identification of vertices of degree $2$ would lead to other examples for which the bound of Theorem~\ref{thm:densitystronger} is either tight or nearly tight.

\begin{figure}[htbp]
	\centering
\begin{minipage}[t]{.3\textwidth}
			\centering
			\begin{tikzpicture}
		[scale=.2]
			\foreach \i in {1,2,3}
			{
				\draw[rotate=120*(\i-1)] (0, 5) node[circle, draw=black!80, inner sep=0mm, minimum size=2.4mm] (x_\i){};
			}
			
			\foreach \i in {1,2,3}
			{
				\draw[rotate=120*(\i-1)-60] (0, 5) node[circle, draw=black!80, inner sep=0mm, minimum size=2.4mm] (z_\i){};
			}
			
			\foreach \i in {1,2,3}
			{
				\draw[rotate=120*(\i-1)] (0, 3) node[circle, draw=black!80, inner sep=0mm, minimum size=2.4mm] (y_\i){};
			}
					
			\foreach \i in {1,2,3}
			{
				\draw[dashed, line width=0.5mm, red] (x_\i) -- (z_\i);
				\draw[line width=0.5mm, blue] (y_\i) -- (z_\i);
	
				}
				
				\foreach \i/\j in {1/2,2/3,3/1}
			{
				\draw[line width=0.5mm, blue] (y_\i) -- (z_\j);
				\draw[line width=0.5mm, blue] (x_\i) -- (z_\j);
				}		
			\end{tikzpicture}
			\caption{$T_{2}(\tilde{C}_{\!\scriptscriptstyle \, 3})$}
			\label{fig:T2C_3}
		\end{minipage}
	\begin{minipage}[t]{.3\textwidth}
		\centering
			\begin{tikzpicture}
			[scale=.2]
			\foreach \i in {1,2,3,4,5}
			{
				\draw[rotate=72*(\i-1)] (0, 5) node[circle, draw=black!80, inner sep=0mm, minimum size=2.4mm] (x_\i){};
			}
			
			\foreach \i in {1,2,3,4,5}
			{
				\draw[rotate=72*(\i-1)-36] (0, 5) node[circle, draw=black!80, inner sep=0mm, minimum size=2.4mm] (z_\i){};
			}
			
			\foreach \i in {1,2,3,4,5}
			{
				\draw[rotate=72*(\i-1)] (0, 3) node[circle, draw=black!80, inner sep=0mm, minimum size=2.4mm] (y_\i){};
			}
					
			\foreach \i in {1,2,3,4,5}
			{
				\draw[dashed, line width=0.5mm, red] (x_\i) -- (z_\i);
				\draw[line width=0.5mm, blue] (y_\i) -- (z_\i);
	
				}
				
				\foreach \i/\j in {1/2,2/3,3/4,4/5,5/1}
			{
				\draw[line width=0.5mm, blue] (y_\i) -- (z_\j);
				\draw[line width=0.5mm, blue] (x_\i) -- (z_\j);
				}		
			\end{tikzpicture}
			\caption{$T_{2}(\tilde{C}_{\!\scriptscriptstyle \, 5})$}
			\label{fig:T2C_5}
			\end{minipage}
\begin{minipage}[t]{.3\textwidth}
		\centering
			\begin{tikzpicture}
			[scale=.2]
			\foreach \i in {1,2,3,4,5}
			{
				\draw[rotate=72*(\i-1)] (0, 5) node[circle, draw=black!80, inner sep=0mm, minimum size=2.4mm] (x_\i){};
			}

			\foreach \i in {1,2,3,4,5}
			{
				\draw[rotate=72*(\i-1)-36] (0, 5) node[circle, draw=black!80, inner sep=0mm, minimum size=2.4mm] (z_\i){};
			}
			
			\foreach \i in {3,4,5}
			{
				\draw[rotate=72*(\i-1)] (0, 3) node[circle, draw=black!80, inner sep=0mm, minimum size=2.4mm] (y_\i){};
			}
			\foreach \i in {1,2}
			{		
                \draw[rotate=36] (0, 3) node[circle, draw=black!80, inner sep=0mm, minimum size=2.4mm] (y_\i){};
            }

			\foreach \i in {1,2,3,4,5}
			{
				\draw[dashed, line width=0.5mm, red] (x_\i) -- (z_\i);
				\draw[line width=0.5mm, blue] (y_\i) -- (z_\i);
	
				}
				
				\foreach \i/\j in {1/2,2/3,3/4,4/5,5/1}
			{
				\draw[line width=0.5mm, blue] (y_\i) -- (z_\j);
				\draw[line width=0.5mm, blue] (x_\i) -- (z_\j);
				}		
			\end{tikzpicture}
			\caption{$\hat{G}'_5$}
			\label{fig:T2C'_5}
			\end{minipage}
		\end{figure}

Another method of building $C_{\Four}$-critical signed graphs is as follows. Let $\hat{G}_1$ and $\hat{G}_2$ be two $C_{\Four}$-critical signed graphs each with a vertex of degree 2. Suppose $u$ is a vertex of degree 2 in $\hat{G}_1$ with $u_1$ and $u_2$ as its neighbors, and $v$ is a vertex of degree 2 in $\hat{G}_2$ with $v_1$ and $v_2$ as its neighbors. As $\hat{G}_1$ is a $C_{\Four}$-critical signed graph, $\hat{G}_1-u$ maps to $C_{\Four}$. But any such mapping must map $u_1$ and $u_2$ to the same vertex of $C_{\Four}$ and must have applied a switching on $\hat{G}_1-u$ so that with the same switching on $\hat{G}_1$, the path $u_1uu_2$ is negative. We consider $\hat{G}_1$ with this signature and do the same on $\hat{G}_2$. We then build a signed graph ${\cal F}(\hat{G}_1,\hat{G}_2)=\hat{G}$ from disjoint union of $\hat{G}_1$ and $\hat{G}_2$ by deleting $u$ and $v$, and adding a positive edge $u_1v_1$ and a negative edge $u_2v_2$. We leave it to the reader to verify that the result is a $C_{\Four}$-critical signed graph. In Figure~\ref{fig:CriticalConstruction},  we have depicted the signed graph obtained from this operation on two disjoint copies of $\hat{W}$. We note that this is an example of a $C_{\Four}$-critical signed graph on 12 vertices for which the bound of Theorem~\ref{thm:densitystronger} is tight. One may note that, furthermore, the same technique can be applied to build a new $C_{\!\scriptscriptstyle -2k}$-critical signed graph from two $C_{\!\scriptscriptstyle -2k}$-critical signed graphs each having a vertex of degree 2. Moreover, towards building a $C_{\!\scriptscriptstyle -2k}$-critical signed graph of lower edge-density, instead of connecting $u_iv_i$ directly, one may use paths of length $k-1$, one of positive sign, one of negative sign.

\begin{figure}[htbp]
\centering
\begin{minipage}[t]{.4\textwidth}
	\centering
		\begin{tikzpicture}
		[scale=.2]
		\foreach \i in {1,2,3,4,5,6}
		{
			\draw[rotate=60*(\i-1)-30] (0, 4) node[circle, draw=black!80, inner sep=0mm, minimum size=2.4mm] (x_\i){};
                 }
                 \foreach \i in {6,8,12,14}
		{
			\draw[rotate=30*(\i-1)-90] (0, 7) node[circle, draw=black!80, inner sep=0mm, minimum size=2.4mm] (y_{\i}){};
                 }
                  \foreach \i in {7,13}
		{
			\draw[rotate=30*(\i-1)-90] (0, 8) node[circle, draw=black!80, inner sep=0mm, minimum size=2.4mm] (y_{\i}){};
		}
		\foreach \i/\j in {4/5, 2/3, 1/6, 3/4, 5/6}
		{
			\draw[line width=0.5mm, blue] (x_\i) -- (x_\j);
		}
		
		\draw[line width=0.5mm, blue] (x_2) -- (y_{6}) -- (y_{7}) -- (y_{8}) -- (x_4);
		\draw[line width=0.5mm, blue] (x_1) -- (y_{14}) -- (y_{13}) -- (y_{12}) -- (x_5);
		\draw[dashed, line width=0.5mm, red] (x_3) -- (y_{7});
		\draw[dashed, line width=0.5mm, red] (x_6) -- (y_{13});
		\draw[dashed, line width=0.5mm, red] (x_1) -- (x_2);
		\end{tikzpicture}
		\caption{${\cal F}(\hat{W}, \hat{W})$}
		\label{fig:CriticalConstruction}
		\end{minipage}
	\begin{minipage}[t]{.4\textwidth}
		\centering
		\begin{tikzpicture}
		[scale=.2]
		\foreach \i in {1,2,3,4,5}
		{
			\draw (4*\i, 4) node[circle, draw=black!80, inner sep=0mm, minimum size=2.4mm] (x_{\i}){};
		}
	
         \foreach \i in{1,2,3,4,5}
		{
			\draw (4*\i, 0 ) node[circle, draw=black!80, inner sep=0mm, minimum size=2.4mm] (y_{\i}){};
		}

\foreach \i/\j in {2/3,3/4,4/5}
			{
				\draw[line width=0.5mm, blue] (x_{\i}) -- (x_{\j});
				}
				
\foreach \i/\j in {1/2,2/3,4/5}
			{
				\draw[line width=0.5mm, blue] (y_{\i}) -- (y_{\j});
				}

	    	\draw[ line width=0.5mm, blue] (x_{1}) -- (y_{1});
            \draw[ line width=0.5mm, blue] (x_{2}) -- (y_{2});
            \draw[ line width=0.5mm, blue] (x_{4}) -- (y_{4});

            \draw[dashed, line width=0.5mm, red] (x_{1}) -- (x_{2});
            \draw[dashed, line width=0.5mm, red] (y_{3}) -- (y_{4});
			\draw[dashed, line width=0.5mm, red] (x_{5}) -- (y_{5});
		    
            \draw[line width=0.5mm, blue] (x_{1}) edge[bend left] (x_{4});
            \draw[line width=0.5mm, blue] (y_{2}) edge[bend right] (y_{5});
		
		\end{tikzpicture}
		\caption{${\cal H}(\Gamma,\Gamma)$}
		\label{fig:HajosConstruction}
	\end{minipage}
\end{figure}

{\bf Analogue of Haj\'{o}s construction.}
The Haj\'{o}s construction of $k$-critical graphs can be adapted to build $C_{\Four}$-critical signed graphs from two given $C_{\Four}$-critical signed graphs. The general case will be addressed in a forthcoming work.  Let $\hat{G}_1$ be a $C_{\Four}$-critical signed graph and let $x_1y_1$ be a positive edge of $\hat{G}_1$. Then $\hat{G}_1-x_1y_1$ admits a homomorphism $\phi$ to $C_{\Four}$. Since $C_{\Four}$ is vertex transitive, and since $(-\phi_1, \phi_2)$ is the same as $(\phi_1, \phi_2)$, we may consider only the mappings for which $\phi(x_1)=(+,u_2)$ (where $u_2$ refers to the labeling of $C_{\Four}$ in Figure~\ref{fig:C_4}). Then we must have $\phi_1(y_1)=-$ as otherwise, $\phi$ is also a mapping of $\hat{G}_1$ to $C_{\Four}$. Furthermore, for any other edge $e$,  if we take a mapping $\phi'$ of $\hat{G}-e$ satisfying $\phi'(x_1)=(+,u_2)$, then we must have $\phi(y_1)=+$.  
Similarly, consider a $C_{\Four}$-critical signed graph $\hat{G}_2$ with a negative edge $x_2y_2$. Then by a similar argument, for any mapping $\psi$ of $\hat{G}_2-x_2y_2$ for which $\psi(x_2)=(+,u_2)$, we must have $\psi(y_2)=+$. 

We now build a new $C_{\Four}$-critical signed graph $\hat{H}={\cal H}(\hat{G}_1,\hat{G}_2)$ as follows: $\hat{H}$ is obtained from vertex disjoint copies of $\hat{G}_1$ and $\hat{G}_2$ by deleting $x_1y_1, x_2y_2$ and identifying $x_1$ with $x_2$ to get a vertex $x$ and $y_1$ with $y_2$ to get a vertex $y$. We observe that if there exists a homomorphism $\varphi$ of $\hat{H}$ to $C_{\Four}$, then, by symmetries, we may assume $\varphi(x)=(+,u_2)$. Then the restriction on $\hat{G}_1$ implies  $\varphi_1(y)=-$ and the restriction on $\hat{G}_2$ implies $\varphi_1(y)=+$, a contradiction, implying that $\hat{H}$ does not map to $C_{\Four}$. Removing an edge from one part of $\hat{H}$ then leads in mappings of the two different parts that can be merged together, which shows that $\hat{H}$ is $C_{\Four}$-critical. An example of this construction, using two disjoint copies of the unique $C_{\Four}$-critical signed graph $\Gamma$ on six vertices (see Figure~\ref{fig:Gamma}) is given in Figure~\ref{fig:HajosConstruction}.

The signed graph ${\cal H}(\hat{G}_1, \hat{G}_2)$ has $|V(\hat{G}_1)|+|V(\hat{G}_2)|-2$ vertices. Using the techniques mentioned above one can easily build $C_{\Four}$-critical signed graphs of orders $9,10,11,12$. Then applying Haj\'{o}s construction to a previously built $C_{\Four}$-critical signed graph and $\Gamma$ (on 6 vertices), one can build a $C_{\Four}$-critical signed graph on any number $n$ of vertices for $n\geq 9$. 

\medskip
Given positive integers $k$ and $n$ ($n\geq k+2$), let $f(n,k)$ be the minimum number of edges of a $k$-critical graph on $n$ vertices. We refer to \cite{KY14} for an almost precise value of $f(n,k)$ and for historical background on the study of this function.
We similarly may define $g(n, k)$ to be the minimum number of edges of a $C_{\CK}$-critical signed graph on $n$ vertices. As noted above, $g(n,4)$ is well-defined for $n\geq 9$. It can be similarly shown that $g(n,k)$ is well-defined for $n\geq N_k$ where $N_k$ is an integer depending on $k$ only. 

Lemma~\ref{lem:T_k-2} and Theorem~\ref{thm:T-of-Tilde} imply the following relations between $f(n,k)$ and $g(n,k)$.

\begin{itemize}
\item By Lemma~\ref{lem:T_k-2},  
\begin{equation}\label{equ:f-g}
g(n+(k-3)f(n,k), k)\leq (k-2)f(n,k).
\end{equation}
\item By Theorem~\ref{thm:T-of-Tilde},

\begin{equation}\label{equ:2f-g}
 g(n+2(l-1)f(n,l), 2l) \leq 2(l-1)f(n,l).
\end{equation}
\end{itemize}

Authors of \cite{DP17} and \cite{PS19} suggest that for $k=5$ and $k=7$  the inequality~(\ref{equ:f-g}) is almost tight. Our work here shows that for $C_{\Four}$-critical signed graphs the inequality of (\ref{equ:2f-g}) provides a tight bound. For $k=6$, the two inequalities provide similar bounds where the only difference is in the constant (in the favor of inequality of (\ref{equ:f-g})). For other values of $k=2l$ the inequality of (\ref{equ:f-g}) provides a better bound than (\ref{equ:2f-g}) and it is tempting to suggest that (\ref{equ:f-g}) gives a nearly tight bound for $g(n,k)$ for $k\geq 5$. A point of hesitation here is that, while the notion of $k$-critical graphs is widely studied and the value and behavior of $f(n,k)$ are almost determined, the notion of critical signed graphs, aside from its relation to $(2k+1)$-critical graphs (with no sign), is a new notion and hardly anything is known about it. In particular, what can then be said about the minimum number of edges of an $X_{2k}$-critical signed graph? Constructions other than $\tilde{G}$, combined with Theorem~\ref{thm:T_k-2}, may provide better bounds on $g(n, 2k)$.

\section{Application to signed bipartite planar  graphs}\label{sec:Planar}

Introducing a bipartite analogue of Jaeger-Zhang conjecture, it was conjectured in $\cite{NRS15}$ that every signed bipartite planar graph, whose shortest negative cycles are of length at least $4k-2$, maps to $C_{\!\scriptscriptstyle -2k}$. In support of the conjecture, the claim is proved, in \cite{CNS20}, for a weaker condition when the negative girth is at least $8k-2$. Here we use the folding Lemma of \cite{NRS13} to prove that every signed bipartite planar graph whose  negative girth is at least $8$ maps to $C_{\Four}$ and we show that this bound is tight, thus disproving the exact claim of the conjecture for the case $k=2$.

\begin{lemma}\cite{NRS13}
\label{lem:folding}
Given a signed bipartite planar graph $(G,\sigma)$ with an embedding on the plane, if the length of the shortest negative cycles of $(G, \sigma)$ is at least $2k$ and a face $F$ is not a negative cycle of length $2k$, then there is a homomorphic image of $(G, \sigma)$ which identifies two vertices at distance 2 of $F$ and such that its shortest negative cycles are also of length at least $2k$.  
\end{lemma}

Observe that this identification preserves both planarity and bipartiteness. Thus, repeatedly applying the lemma, we get a homomorphic image where all faces are negative cycles of length $2k$. Taking $k=4$, starting from a signed bipartite planar graph whose shortest negative cycles are of length at least $8$, we get a homomorphic image $\hat{G}$ with a planar embedding where all faces are (negative) 8-cycles. Applying the Euler formula on this graph, we have $|E(G)|\leq \frac{4}{3}(|V(G)|-2)$. By taking $\hat{G}$ to be a smallest signed bipartite planar graph which does not map to $C_{\Four}$ and whose shortest negative cycle is of length 8, we conclude that on the one hand $\hat{G}$ must be $C_{\Four}$-critical, and thus, by Theorem~\ref{thm:densitystronger}, has at least $\frac{4}{3}|V(G)|$ edges, but on the other hand, by the argument above, it has at most $\frac{4}{3}(|V(G)|-2)$ edges. This contradiction is a proof that:

\begin{theorem}\label{thm:girth8}
Any signed bipartite planar graph of negative girth at least 8 maps to $C_{\Four}$.
\end{theorem}

We now claim that the condition of shortest negative cycles being of length at least 8 in this theorem is tight. 

For this, it would be enough to build a signed planar (simple) graph $(G, \sigma)$ which is not $\{\pm1, \pm2\}$-colorable. Then, by Theorem~\ref{thm:T_k-2}, $T_2(G, \sigma)$ is a signed bipartite planar graph which does not map to $C_{\Four}$. Furthermore, that $G$ is simple implies that $T_2(G, \sigma)$ has no cycle of length smaller than 6. 

That every signed planar graph is $\{\pm1, \pm2\}$-colorable was conjectured in \cite{MRS16}. This conjecture was disproved in \cite{KN21}, we refer to \cite{NP21+} for a direct proof. Thus we have:

\begin{theorem}
There exists a bipartite planar graph $G$ of girth 6 with a signature $\sigma$ such that $(G,\sigma) \not\to C_{\Four}.$
\end{theorem}

The smallest examples we have built so far in this way have $150$ vertices. However, such examples have the extra property that vertices on one part of the (bipartite) graph are all of degree 2. Perhaps simpler examples can be built which do not satisfy this property.

It is proved in \cite{DFMOP20} that $C_{\Four}$-coloring problem even when restricted to the class of signed (bipartite) planar graphs remains an NP-complete problem. Thus, one does not expect to find an efficient classification of signed bipartite planar graphs which map to $C_{\Four}$. However, some strong sufficient conditions could be provided. One such condition is based on the restatement of the 4CT given in Theorem~\ref{thm:4CT-restated}. Another is Theorem~\ref{thm:girth8} of this work that shows no negative cycle of length $2,4,6$ is a sufficient condition. As a generalization of Theorem~\ref{thm:4CT-restated} (the  4CT) which also captures essential cases of Theorem~\ref{thm:girth8}, we propose the following:

\begin{conjecture}
Let $G$ be a bipartite planar graph of girth at least 6. Let $\sigma$ be a signature on $G$ such that in $(G, \sigma)$ all 6-cycles are of the same sign. Then $(G, \sigma)\to C_{\Four}$. 
\end{conjecture}

We note that, while one may use Lemma~\ref{lem:folding} to reduce facial 4-cycles of a signed graph which is the subject of Theorem~\ref{thm:girth8}, there could be separating 4-cycles in a signed bipartite planar graph to which this theorem may apply. Therefore, the conjecture does not capture all cases to which Theorem~\ref{thm:girth8} applies.

As a final remark, we would like to point out that some of the results in this work can be restated using the language of the circular coloring of signed graphs which is recently developed in \cite{NWZ21}.\\

{\bf Acknowledgment.} This work is supported by the ANR (France) project HOSIGRA (ANR-17-CE40-0022). It has also received funding from the European Union's Horizon 2020 research and innovation program under the Marie Sklodowska-Curie grant agreement No 754362. 

We would like to thank the referees and E. Sopena for carefully reading this work and for helping us to improve.

\bibliographystyle{acm}
\bibliography{CriticalGraph}

\end{document}